\documentclass[reqno]{amsart}
\usepackage[latin1]{inputenc}
\usepackage[T1]{fontenc}
\usepackage{amssymb, amsmath, color, textcomp, url,tikz}

\usepackage{pifont}

\usetikzlibrary{matrix,arrows}
\usetikzlibrary{fit}
\usetikzlibrary{positioning}
\usepackage[labelformat=empty]{caption}
\usepackage{mdwlist}
\usepackage{etoolbox}

\RequirePackage{ifpdf}
\ifpdf
   \usepackage[pdftex]{hyperref}
\else
   \usepackage[hypertex]{hyperref}
\fi

\usepackage{enumerate,xspace}

\theoremstyle{plain}
\newtheorem{theorem}{Theorem}[section]
\newtheorem{cor}[theorem]{Corollary}
\newtheorem{prop}[theorem]{Proposition}
\newtheorem{lemma}[theorem]{Lemma}

\newcounter{claimCount}
\newenvironment{claim}{\medskip \vskip-1mm\noindent
  \refstepcounter{claimCount}\textbf{Claim~\arabic{claimCount}.}}{\medskip}

\newenvironment{claimproof}{\noindent\textit{Proof of
    Claim~\arabic{claimCount}.}}{\hfill\qedsymbol
    \tiny{Claim~\arabic{claimCount}}
\medskip}

\AtBeginEnvironment{proof}{\setcounter{claimCount}{0}}

\newtheorem*{teo}{Main Theorem}

\theoremstyle{definition}
\newtheorem{remark}[theorem]{Remark}
\newtheorem{fact}[theorem]{Fact}
\newtheorem{definition}[theorem]{Definition}

\newtheorem*{question}{Question}
\newtheorem*{notation}{Notation}

\newcommand{\nc}{\newcommand}

\nc{\Z}{\mathbb{Z}}
\nc{\N}{\mathbb{N}}
\nc\LL{\mathcal L}
\nc\LLD{\mathcal L_{D}}
\nc\FF{\mathcal F}
\nc\UU{\mathbb U}

\nc\CC{\mathcal C}
\nc\TC{\mathcal T_\CC}

\nc{\acfp}{\mathsf{ACFP}}

\nc{\Land}{\bigwedge}
\nc{\XX}{X_0,\dots,X_n}
\nc{\hs}{\mathcal H^n_Q}
\nc{\Hs}{\mathrm H^n_Q}

\renewcommand{\phi}{\varphi}

\DeclareMathOperator{\dep}{Ind}

\DeclareMathOperator{\RM}{RM}
\DeclareMathOperator{\smallrm}{rm}
\DeclareMathOperator{\RU}{U}
\DeclareMathOperator{\Rirr}{\mathrm{Rk}_{\CC}}
\DeclareMathOperator{\RF}{R_{\FF}}
\DeclareMathOperator{\I}{I}
\DeclareMathOperator{\V}{V}

\nc\Cb{\operatorname{Cb}}

\nc\lc{$\lambda$-closed\xspace}
\nc\Hil{\operatorname{Hil}}

\nc\ord{\operatorname{ord}}
\nc{\dcl}{\operatorname{dcl}}
\nc{\acl}{\operatorname{acl}}
\nc{\tr}{\operatorname{tr}}
\nc{\codim}{\operatorname{codim}}
\nc\inv{ ^{-1}}
\def\scl#1{\langle#1\rangle}

\def\cp #1{\ulcorner #1\, \urcorner}

\nc{\tp}{\operatorname{tp}}
\nc{\cf}{\text{cf. }}

\nc{\Sat}{\operatorname{Sat}}
\nc\Ann{\operatorname{Ann}}
\nc{\Gr}{\operatorname{Gr}}
\nc{\pk}{\operatorname{Pk}}
\nc{\pluck}{\operatorname{\lrcorner}}

\makeatletter
\newcommand{\extp}{\@ifnextchar^\@extp{\@extp^{\,}}}
\def\@extp^#1{\mathop{\Land\nolimits^{\!#1}}}
\makeatother

\def\Ind#1#2{#1\setbox0=\hbox{$#1x$}\kern\wd0\hbox to 0pt{\hss$#1\mid$\hss}
\lower.9\ht0\hbox to 0pt{\hss$#1\smile$\hss}\kern\wd0}
\def\Notind#1#2{#1\setbox0=\hbox{$#1x$}\kern\wd0\hbox to
0pt{\mathchardef\nn="0236\hss$#1\nn$\kern1.4\wd0\hss}\hbox
to 0pt{\hss$#1\mid$\hss}\lower.9\ht0
\hbox to 0pt{\hss$#1\smile$\hss}\kern\wd0}
\def\ind{\mathop{\mathpalette\Ind{}}}
\def\nind{\mathop{\mathpalette\Notind{}}}

\def\ld{\mathop{\ \ \hbox to 0pt{\hss$\mid^{\hbox to
0pt{$\scriptstyle\mathrm{ld}$\hss}}$\hss}
\lower4pt\hbox to 0pt{\hss$\smile$\hss}\ \ }}

\begin{document}

\title{Noetherian theories}

\date{August 10, 2024}

\author{ Amador Martin-Pizarro and Martin Ziegler}
\address{Mathematisches Institut,
  Albert-Ludwigs-Universit\"at Freiburg, D-79104 Freiburg, Germany}
\email{pizarro@math.uni-freiburg.de}
\email{ziegler@uni-freiburg.de}
\thanks{Research partially supported by the program
PID2020-116773GB-I00. Additionally, the first author conducted
research supported by the Deutsche
  Forschungsgemeinschaft (DFG, German Research Foundation) -
  Project number 431667816, part of the
  ANR-DFG, program GeoMod. }
\keywords{Model Theory, Noetherianity, Equationality, Pairs of
fields}
\subjclass{03C45}

\begin{abstract}
  A first-order theory is Noetherian with respect to the collection of
  formulae $\FF$ if every definable set is a Boolean combination of
  instances of formulae in $\FF$ and the topology whose subbasis of
  closed sets is the collection of instances of arbitrary formulae in
  $\FF$ is Noetherian. We show the Noetherianity of the theory of
  proper pairs of algebraically closed fields in any characteristic
  with respect to the family of tame formulae as introduced in
  \cite{MPZ20}, thus answering a question which was left open there.
\end{abstract}

\maketitle

\section{Introduction}

Consider a first-order theory $T$ and a collection of formulae $\FF$
closed under finite conjunctions. The family $\FF$ is
\emph{Noetherian} if in every model $M$ of $T$ the family of instances
of arbitrary formulae in $\FF$ has the descending chain condition. The
theory $T$ is \emph{Noetherian} with respect to the Noetherian
collection of formulae $\FF$ if every formula $\phi(x; y)$ (in a fixed
partition of the variables into tuples $x$ and $y$) is equivalent
modulo $T$ to a Boolean combination of formulae $\psi(x; y)$ in $\FF$
(in the same partition).

After having submitted a first draft of this article, we were made
aware that Noetherian theories had already been defined by Hoffmann
and Kowalski \cite{HK23}, albeit in a slightly different formulation.
We will discuss the equivalent formulations in Remark \ref{R:HK}.

Quantifier elimination implies that the theory of algebraically closed
fields as well as the theory of differentially closed fields of
characteristic $0$ are Noetherian, since both the Zariski and the
Kolchin topologies are Noetherian. Every differentially closed field (in
characteristic $0$) expands the structure of a proper pair of
algebraically closed fields, where the distinguished algebraically
closed subfield is given by the \emph{constant} elements, whose
derivative is $0$.

In \cite{aG16, MPZ20} it was shown that proper pairs of algebraically
closed fields of characteristic $0$ are Noetherian. This follows from
the fact that definable sets in the pair are Boolean combination of
certain definable sets which happen to be Kolchin-closed in the
corresponding expansion as a differentially closed field. However,
this approach cannot be carried over to the case of positive
characteristic since the Kolchin topology for differentially closed
fields of positive characteristic is not Noetherian.

A weakening of Noetherianity is equationality, in which we only
require that each partitioned formula is a boolean combination of
\emph{equations} (in the same partition). A partitioned formula
$\phi(x; y)$ is an equation if in every model of $T$ the family of
finite intersections of instances $\phi(x,a)$ has the descending chain
condition. The authors showed in \cite{MPZ20} that the theory of pairs
of algebraically closed fields is equational.

In this paper, we will prove that the family of equations exhibited in
\cite{MPZ20} for the theory of proper pairs of algebraically closed
fields is in fact Noetherian, so every proper pair of algebraically
closed fields is Noetherian, regardless of the characteristic (\cf
Section 9 of the extended version of \cite{MPZ20}).

\begin{teo}\textup{(}Corollary  \ref{C:acfp_noether}\textup{)}~
  The theory of proper pairs of algebraically closed fields is
  Noetherian.
\end{teo}

The structure of the papers is as follows: In Section \ref{S:noether},
we explore the notion of Noetherianity. It will follow from Corollary
\ref{C:min} that Noetherianity is equivalent to the fact that every
type contains a minimal instance of a formula in $\FF$. Moreover, we
relate Morley rank to the foundational rank relative to $\FF$ and show
that equality holds under a mild condition, called \emph{Noetherian
isolation}. Section \ref{S:belles} contains a short overview of the
main properties of the theory of proper pairs of algebraically closed
fields, which will be used in Section \ref{S:noether_trick} in order
to give a proof of the Noetherianity of this theory. In Section
\ref{S:RM} we show that the theory of proper pairs of algebraically
closed fields has Noetherian isolation using Poizat's description of
Morley and Lascar ranks. Finally, in Section \ref{S:Hilbertpol}, we
use the techniques of Hilbert polynomials and schemes (in a
self-contained presentation) in order to explicitly exhibit the
minimal tame formula of a type.

\medskip
\noindent {\bf Acknowledgements.} We would like to express our sincere
gratitude to the anonymous referee of a first version of this draft
for the helpful comments and remarks which have clearly improved (in our
opinion) the presentation. We would also like to thank Martin Hils,
Piotr Kowalski and Frank O. Wagner for their comments which we have
tried to address as much as possible in this revised version.

\section{Noetherianity and chain conditions}\label{S:noether}

\begin{definition}
  A collection $\CC$ of subsets of a set $X$ is \emph{Noetherian} if
  it satisfies the following two conditions:
  \begin{itemize}
  \item The collection $\CC$ is closed under finite intersections and
    contains the set $X$ itself.
  \item The collection $\CC$ has the descending chain condition (DCC):
    every descending chain \[ C_0\supset C_1\supset \dotsb \supset
    C_n\supset \dotsb, \] with $C_n$ in $\CC$ for $n$ in $\N$,
    eventually stabilises, that is, there is some $n_0$ such that
    $C_n=C_{n+1}$ for all $n\ge n_0$.
  \end{itemize}
\end{definition}
It is easy to see that $\CC$ is has DCC if and only if every non-empty
subset of $\CC$ has a minimal element (with respect to set-theoretic
inclusion).

\begin{lemma}\label{L:ultra_Noether}
  Consider a collection $\CC$ of subsets of a set $X$ such that $\CC$
  contains $X$ and is closed under finite intersections. Then the
  following are equivalent;
  \begin{enumerate}[(a)]
  \item\label{I:ultra_Noether:dcc} The collection $\CC$ is Noetherian.
  \item\label{I:ultra_Noether:minimal} For every ultrafilter $\mathcal
    U$ on $X$, the intersection $\mathcal U\cap\CC$ has a minimal
    element $D$.
  \item\label{I:ultra_Noether:trick} Every ultrafilter $\mathcal U$ on
    $X$ contains a set $Y$ (possibly not in $\CC$) which is contained
    in every element of $\mathcal U\cap\CC$.
  \end{enumerate}
\end{lemma}
\noindent Since $\mathcal U\cap\CC$ is closed under finite
intersections, the subset $D$ in (\ref{I:ultra_Noether:minimal}) is
uniquely determined. We refer to $D$ as \emph{the minimal element of\/
$\mathcal U$ with respect to} $\CC$.

\begin{proof}
  The implications
  $(\ref{I:ultra_Noether:dcc})\Rightarrow(\ref{I:ultra_Noether:minimal})$
  and
  $(\ref{I:ultra_Noether:minimal})\Rightarrow(\ref{I:ultra_Noether:trick})$
  are immediate. For the implication
  $(\ref{I:ultra_Noether:trick})\Rightarrow
  (\ref{I:ultra_Noether:dcc})$, consider a strictly decreasing
  chain \[ C_0\supsetneq C_1\supsetneq\dotsb \supsetneq C_n\supsetneq
  \dotsb, \] of elements of $\CC$. Set $Z=\bigcap\limits_{n\in \N}
  C_n$ and notice that the collection $\{C_n\setminus Z\}_{n\in \N} $
  has the finite intersection property. Thus, there is some
  ultrafilter $\mathcal U$ on $X$ containing every $C_n\setminus Z$.
  Assume that $\mathcal U$ contains an element $Y$ as in
  (\ref{I:ultra_Noether:trick}). Since $\mathcal U\cap\CC$ has empty
  intersection, we deduce that $Y=\emptyset$ which gives the desired
  contradiction.
\end{proof}

\begin{definition}\label{D:irred}
  An element $C$ of $\CC$ is \emph{irreducible} if it is non-empty and
  cannot be written as a finite union $C=C_1\cup\dotsb\cup C_n$, with
  each $C_k\subsetneq C$ in $\CC$. Equivalently, whenever $C$ is
  contained in some finite union $\bigcup_{i=1}^m D_i$, with $D_i$ in
  $\CC$, then $C\subset D_i$ for some $1\le i\le m$.
\end{definition}

\begin{remark}\label{R:irred}
  If $\CC$ is Noetherian, it follows immediately from K\"onig's Lemma
  that every $C$ in $\CC$ can be written as an irredundant union of
  finitely many irreducible subsets $C_1, \dotsc, C_n$. The
  decomposition $C=C_1\cup\dotsb \cup C_n$ is \emph{irredundant} if
  $C_i\not\subset C_j$ for $i\ne j$. The irreducible subsets appearing
  in an irredundant expression of $C$ are unique up to permutation. We
  refer to them as the \emph{irreducible components} of $C$.
\end{remark}

A straightforward application of Lemma \ref{L:ultra_Noether} yields
the following result, which justifies our choice of terminology.

\begin{fact}\textup{(}\cite[Lemma
2.7]{PiSr84}\textup{)}~ If the collection $\CC$ of subsets of $X$ is
  Noetherian, then so is the family $\CC'$ of finite unions of
  elements of $\CC$.
\end{fact}
\noindent A topology is \emph{Noetherian} if the family of closed sets
is Noetherian. If $\CC$ is Noetherian, it follows that the family $\CC'$ of finite unions of
elements of $\CC$ consists of the closed sets of a Noetherian topology $\TC$ on $X$.

\begin{definition}\label{D:Fund_Rang}
  Given a Noetherian collection $\CC$ of subsets of $X$, we assign an
  ordinal rank $\Rirr(Y)$ to every subset $Y$ of $X$ as follows. For
  irreducible sets $C$ it is the foundational rank, that is,
  \begin{itemize}
  \item $\Rirr(C)\ge 0$ always holds (for the irreducible set $C$ is not empty);
  \item $\Rirr(C)\ge \alpha+1$ if and only if $\Rirr(D)\ge \alpha$ for
    some irreducible $D\subsetneq C$;
  \item $\Rirr(C)\ge \alpha$ with $\alpha$ limit if and only if
    $\Rirr(C)\ge \beta$ for every $\beta<\alpha$.
  \end{itemize}
  We set $\Rirr(C)$ the largest $\alpha$ with $\Rirr(C)\ge\alpha$.
  (Such an ordinal always exist by Noetherianity of the family). The
  rank of a closed set is the largest rank of its irreducible
  components, whilst $\Rirr(\emptyset)=-\infty$. Finally the rank of
  an arbitrary subset $Y$ is the rank of its closure $\overline Y$
  with respect to the topology $\TC$.
\end{definition}

\begin{remark}\label{R:Fundrang_abg}
  If $A$ is closed in $X$ with respect to $\TC$, we have that
  \[ \Rirr(A)=\max \{ \Rirr(C) \
  | \ C\subset A \text{ irreducible } \}.\] This follows from the fact
  that every irreducible subset of $A$ is contained in an irreducible
  component of $A$. Using the above equality, we deduce that \[
  \Rirr(Y_1\cup Y_2)=\max\{\Rirr(Y_1), \Rirr(Y_2)\}\] for any subsets
  $Y_1$ and $Y_2$ of $X$. Now, if the subset $Y$ of $X$ is
  \emph{constructible}, that is, it is a Boolean combination of closed
  sets, write $Y=\bigcup_{1\le i\le n} C_i\cap O_i$ for some
  irreducible closed subsets $C_i$ and open subsets $O_i$ with each
  $C_i\cap O_i\ne \emptyset$. It follows that that \[
  \Rirr(Y)=\max\limits_{1\le i\le n} \Rirr(C_i),\] since the closure
  of each $C_i\cap O_i$ is $C_i$.
\end{remark}
The following lemma will be used in the proof of
\ref{L:rk_formel_type}.
\begin{lemma}\label{L:Rang_des_Randes}
  If $Y$ is constructible and non-empty, then $\Rirr(\overline
  Y\setminus Y)<\Rirr(Y)$.
\end{lemma}
\begin{proof}
  Write $Y=\bigcup_{1\le i\le n} C_i\cap O_i$ as in Remark
  \ref{R:Fundrang_abg} and notice that \[ \overline Y \setminus
  Y\subset \bigcup_{1\le i\le n} C_i\setminus O_i.\] Since
  $C_i\setminus O_i$ is a proper closed subset of $C_i$, we have that
  $\Rirr(C_i\setminus O_i) <\Rirr(C_i)$, which gives the desired
  inequality.
\end{proof}

\begin{definition}
  The \emph{degree} of a closed subset $A$ of $X$ is the number
  $\deg_\CC(A)$ of irreducible subsets of $A$ of rank the rank of of
  $A$. The degree of an arbitrary subset of $X$ is the degree of its
  closure.
\end{definition}

The following observation follows from the fact that an irreducible
subset of $\overline{Y^1}\cup\overline{Y^2}$ is contained in
$\overline{Y^1}$ or in $\overline{Y^2}$.
\begin{lemma}
  If $\Rirr(Y^1)> \Rirr(Y^2)$, then $\deg_\CC(Y^1\cup
  Y^2)=\deg_\CC(Y^1)$.\qed
\end{lemma}

Note that the degree of a constructible set $Y=\bigcup_{1\le i\le n}
C_i\cap O_i$ equals the number of different $C_i$'s of maximal rank.
This yields the following result.

\begin{lemma}\label{L:deg_Sum}
  Given two disjoint constructible subsets $Y^1$ and $Y^2$ of $X$ of
  the same rank, we have $\deg_\CC(Y_1\cup Y_2)=\deg_\CC(Y^1) +
  \deg_\CC(Y^2)$.
\end{lemma}

\begin{proof}
  For $j$ in $\{1, 2\}$, write $Y^j=\bigcup_{1\le i\le n_i} C_i^j\cap
  O^j_i$ for some irreducible closed subsets $C^j_i$ and open subsets
  $O^j_i$ with $C^j_i\cap O^j_i\ne \emptyset$. We need only show that
  $C^1_i\neq C^2_k$ for all $i,k$. Assume otherwise. Since
  $C=C^1_i=C^2_k$ is not the union of the two closed proper subsets
  $C\setminus O^1_i$ and $C\setminus O^2_k$, it follows that $C\cap
  O^1_i$ and $C\cap O^2_k$ cannot be disjoint, which gives the desired
  contradiction.
\end{proof}
\bigskip

Fix now a  first-order theory $T$ in a language $\LL$.

\begin{notation}
  Consider a collection of partitioned formulae $\FF$ closed under
  renaming of variables, and finite conjunctions. For simplicity, we
  will always assume that the tautologically true \emph{sentence}
  $\top$ belongs to $\FF$. We allow dummy free variables, so $\top$
  may be considered as a formula in any partitioned set of variables.
\end{notation}

\begin{definition}\label{D:noether_F}
  The collection $\FF$ is \emph{Noetherian} if in every model $M$ of
  $T$ and for every length $n=|x|$, the family of instances \[ \CC=\{
  \phi(M, a) \ | \ \phi(x, y) \in \FF \ \& \ a\in M\} \] is
  Noetherian. We call a definable set $\theta(x,b)$ \emph{closed} if
  $\theta(M,b)$ belongs to $\CC$, that is, if $\theta(M,b)$ equals
  $\phi(M, a)$ for some $\phi(x,y)$ in $\FF$ and $a$ in $M$.
\end{definition}

If the theory $T$ is complete, it suffices to check that the family of
instances with parameters in some $\aleph_0$-saturated model has the
descending chain condition.

\begin{remark}\label{R:noeth_eq}
  Recall that a formula $\phi(x, y)$ is an \emph{equation} if the
  collection of finite intersections of instances of $\phi(x, y)$ has
  the DCC, or equivalently, if the collection of all conjunctions
  $\Land_{i=1}^n \phi(x, y_i)$ is Noetherian. Every formula in a
  Noetherian family $\FF$ is an equation.
\end{remark}
If $M$ is a model of $T$, every ultrafilter on $M^{|x|}$ determines a
type $p(x)$ over $M$, and thus over any subset $A$ of $M$. Hence, we
deduce from Lemma \ref{L:ultra_Noether} and the observation after
Definition \ref{D:noether_F} the following result.
\begin{cor}\label{C:min}
The following conditions are equivalent:
  \begin{enumerate}[(a)]
  \item The collection $\FF$ is Noetherian.
  \item Every type $p(x)$ over a model $M$ of $T$ contains a
    \emph{minimal formula} $\phi(x, a)$ with respect to $\FF$, that
    is, the formula $\phi(x, y)$ belongs to $\FF$ and \[ \psi(x, b)
    \text{ belongs to } p \text{ if and only if } \phi(M, a)\subset
    \psi(M, b)\] for every $\psi(x,z)$ in $\FF$ and every tuple $b$ in
    $M$.
    \item Every type $p(x)$ over a model $M$ of $T$ contains an
      $\LL_M$-formula $\theta(x, a)$ such that
  \[ \psi(x, b) \text{ belongs to } p
  \text{ if and only if } \theta(M, a)\subset \psi(M, b)\] for every
  formula $\psi(x, y)$ in $\FF$ and every tuple $b$ in $M$.
  \end{enumerate} \qed
\end{cor}
Since every two minimal formulae in the type $p$ over $M$ are
equivalent, we will say that $\phi(x, a)$ is \emph{the} minimal
formula of $p$ (with respect to the Noetherian family $\FF$). In
condition (c), we do not require that $\theta(x, y)$ belongs to $\FF$,
so a type may admit two non-equivalent formulae $\theta(x, a)$ and
$\theta'(x, a')$ as in (c).

\begin{remark}\label{R:minimal_A}
  If $\FF$ is Noetherian, it is easy to see that every type $p(x)$
  over a subset $A$ of a model $M$ contains a closed formula $\psi(x,
  a)$, which is minimal among all closed formulae in $p$. We will
  refer to $\psi(x, a)$ as the \emph{minimal formula} of $p$.

If $A$ is an arbitrary subset of parameters and not necessarily an
elementary substructure of $M$, it need not be the case that the
minimal formula of $p$ is of the form $\phi(x,a)$ for some
$\phi(x,y)\in\FF$ and $a$ in $A$. The easiest example is the theory of
a $2$-element set with $\FF$ the family generated by $(x\doteq y)$.
For those readers who do not feel at ease with finite models (which is
the case of the the first author), we provide a more \emph{standard}
example: Consider the theory of a structure with two infinite
equivalence classes modulo a definable equivalence relation $E(x,y)$
and set $\FF$ the family generated by $\{(x\doteq y), E(x,y)\}$. This
family is clearly Noetherian by Corollary \ref{C:min}, since there are
only finitely many atomic formulae over any subset of parameters. If
$a$ is any element, the closed formula $\neg E(x, a)$ is clearly
invariant over $A=\{a\}$, yet it is not an instance over $A$ of an
$\FF$-formula.

We will see in Proposition \ref{P:tame_instance} that minimal closed
formulae for the theory of proper pairs of algebraically closed fields
are indeed equivalent to instances of tame formulae with the same
parameters.
\end{remark}

Using Definition \ref{D:irred}, we can define whether a closed formula
$\psi(x, a)$ in the model $M$ is irreducible. More generally, given a
subset $A$ of some model $M$ of $T$, we say that a closed formula with
parameters in $A$ is \emph{irreducible over $A$} if it cannot be
written as a proper union of a finite number of closed formulas with
parameters in $A$. If $A=M$, we will simply say that $\psi(x, a)$ is
irreducible.

\begin{remark}
  Let $A$ be a subset of a model $M$ of $T$. The minimal formula of a
  type over $A$ is irreducible over $A$.

  A closed formula with parameters in $M$ is irreducible over $M$ if
  and only it is irreducible over any elementary extension of $M$.
  Moreover, if $\theta(x,a)$ is any formula with parameters in $M$,
  then a closed formula $\phi(x,b)$ equals the topological closure
  $\overline{\theta(x,a)}$ of $\theta(x,a)$ in $M$ if and only
  $\phi(x,b)$ is the closure of $\theta(x,a)$ in any elementary
  extension of $M$. It follows that $\phi(x,b)$ can actually be
  defined using the same tuple $a$ of parameters.
\end{remark}

\begin{notation}
  Using Definition \ref{D:Fund_Rang}, given a formula $\theta(x, a)$
  with parameters in a model $M$ of $T$, we denote by $\RF\theta(x,
  a)$ the $\Rirr$-rank of the set $\theta(N,a)$ with respect to the
  Noetherian family $ \CC=\{ \phi(N, b) \ | \ \phi(x, y) \in \FF \ \&
  \ b\in N\}$, where $N$ is some $\aleph_0$-saturated elementary
  extension of $M$. We define the degree $\deg_\FF\theta(x, a)$
  similarly. The rank $\RF(p)$ of a type is the smallest rank of a
  formula in $p$. The degree $\deg_\FF(p)$ is the smallest degree of a
  formula in $p$ of rank $\RF(p)$.
\end{notation}
\noindent Since the closure of a formula has the same rank and degree,
it is easy to see that the rank and the degree of a type are exactly
the rank and the degree of its minimal formula. Whence, the degree of
a a type $p$ over a model is always $1$, since its minimal formula is
irreducible.

\begin{lemma}\label{L:rk_formel_type}
  Given a Noetherian family $\FF$, let $\theta(x,a)$ be a formula with
  parameters in a subset $A$ of a model $M$ of $T$. Then
  \[\RF\theta(x,a)=\max\{\RF(p)\mid \text{the type $p$ over $A$
    contains $\theta(x,a)$}\},\] where $\max \emptyset=-\infty$.
\end{lemma}
\begin{proof}
  This follows easily from Remark \ref{R:Fundrang_abg}: If $\alpha$ is
  the rank of $\theta(x,a)$ (so $\theta(x,a)$ is in particular
  consistent), then the set $\Sigma(x)$ of all negations of formulae
  of rank $<\alpha$ together with $\theta(x,a)$ is consistent. Any
  type over $A$ which extends $\Sigma(x)$ has rank $\alpha$.
\end{proof}

\begin{definition}
  A first-order theory $T$ is \emph{Noetherian} with respect to the
  Noetherian family of formulae $\FF$ if every partitioned formula
  $\psi(x,y)$ is equivalent modulo $T$ to a Boolean combination of
  formulae $\phi(x,y)$ in $\FF$.
\end{definition}

\begin{remark}\label{R:HK}
  In \cite{HK23} Hoffmann and Kowalski defined $T$ to be Noetherian
  with respect to a set $\Sigma$ of formulae if the following two
  conditions hold in every model $M$:
  \begin{itemize}
\item The collection of definable sets $\sigma(M,a)$, with $\sigma(x,
  y)$ in $\Sigma$ and $a$ in $M$, form the class of closed sets of a
  Noetherian topology on $M^{|x|}$.
\item For every subset $A$ of parameters, every definable subset of
  $M^{|x|}$ over $A$ is a Boolean combination of instances
  $\sigma(M,a)$, with $\sigma(x,y)$ in $\Sigma$ and $a$ in $A$.
  \end{itemize}
It is easy to see that both notions are equivalent: Given $\FF$, set
$\Sigma$ to be all finite disjunctions of formulae in $\FF$.
Conversely, given $\Sigma$ as above, let $\FF$ be the collection of
formulae of the form $\sigma(x,y)\land\psi(y)$, where $\sigma$ is in
$\Sigma$ and $\psi$ is arbitrary.
\end{remark}

\begin{remark}\label{R:Tnoether_eq}
  Recall that a theory $T$ is \emph{equational} if every partitioned
  formula is equivalent to a boolean combination of equations. We
  conclude from Remark \ref{R:noeth_eq} that Noetherian theories are
  equational.
\end{remark}
\begin{question}
As pointed out in \cite[p.\ 830]{MPZ20}, a theory is equational if and
only if every completion is. We do not know whether the same holds for
Noetherianity.
\end{question}

We will now explore some of the model-theoretic properties of theories
which are Noetherian and determine their stability spectrum. For that,
we will adapt the previous notions of rank and degree to formulae in
terms of their underlying definable set. It is easy to see that these
definitions do not depend on the choice of the model. Even if no
parameters occur in $\theta(x)$, we still need to choose an ambient
model $M$.

\begin{lemma}\label{L:noether_tt}
  Noetherian theories are totally transcendental.
\end{lemma}

\begin{proof}
  Assume for a contradiction that in some model $M$ of $T$ there is a
  binary tree of consistent formulae $\theta_s(x, a_s)$, with $a_s$ in
  $M$ for $s$ in $\mathbin{^{<\omega}2}$. Since $T$ is Noetherian,
  each definable set $\theta_s(x, a_s)$ is constructible. Choose thus
  an instance $\theta_s(x, a_s)$ in the tree whose $\RF$-rank and
  degree are least possible in the lexicographic order. By minimality
  of the rank, both $\theta_{s\frown 0}(x, a_{s\frown 0})$ and
  $\theta_{s\frown 1}(x, a_{s\frown 1})$ must have the same $\RF$-rank
  as $\theta_s(x, a_s)$. Now, the instances $\theta_{s\frown 0}(x,
  a_{s\frown 0})$ and $\theta_{s\frown 1}(x, a_{s\frown 1})$ are
  disjoint, so we deduce from Lemma \ref{L:deg_Sum} that the degree of
  $\theta_s(x, a_s)$ is strictly larger than $\deg_\CC(\theta_{s\frown
    0}(x, a_{s\frown 0}))$, which gives the desired contradiction.
\end{proof}

It follows that Noetherian theories are $\kappa$-stable for every
$\kappa\ge |\LL|$. We will provide a direct proof of this in terms of
the natural correspondence between types and their minimal formulae.

\begin{prop}\label{P:noether_omega}
  Suppose that the first-order theory $T$ is Noetherian with respect
  to $\FF$. For every subset $A$ of a model $M$ of $T$, there is
  bijection between types over $A$ and (equivalence classes with
  respect to logical equivalence of) irreducible formulas over $A$. In
  particular, every Noetherian theory is $\kappa$-stable whenever
  $\kappa\ge |\LL|$.
\end{prop}
\begin{proof}
  By Remark \ref{R:minimal_A}, given a type $p(x)$ over $A$, we denote
  its minimal formula by $\phi_p(x, a)$, which is unique up to
  equivalence and irreducible over $A$.

  Given now a closed formula $\phi(x, a)$ with parameters from $A$,
  the collection \[ \Sigma_\phi=\{\phi(x, a)\}\cup \{\neg\psi(x, a')
  \ | \ \psi(x, a') \text{ closed, $a'$ in $A$ and } \phi(M,
  a)\not\subset \psi(M, a')\} \] is consistent exactly if $\phi(x, a)$
  is irreducible over $A$. In that case, it admits a unique completion
  $p_\phi(x)$, since $A$-definable subsets are Boolean combinations of
  $A$--definable closed subsets.

  To conclude, we need only observe that $\phi(x, a)$ is the minimal
  formula of a type $p$ over $A$ if and only if $\Sigma_\phi \subset
  p$.
\end{proof}

It is not hard to see that the type $p(x)$ can be recovered from its
minimal formula $\phi(x,a)$ as the set of all formula $\theta(x,a')$
over $A$ such $\phi(x,a)$ is the topological closure of
$\phi(x,a)\land\theta(x,a')$.

\begin{remark}\label{R:Hils}
  In contrast to Remark \ref{R:Tnoether_eq} and Lemma
  \ref{L:noether_tt}, totally transcendental equational complete
  theories need not be Noetherian as the following example shows:
  Bonnet and Si-Kaddour \cite{BR00} constructed a superatomic Boolean
  algebra $B$ of Cantor-Bendixson rank $3$ which is not generated by
  any well-founded subset closed under intersections. Such a Boolean
  algebra must necessarily be uncountable and have rank at least 3.

  We will now consider the language $\LL$ consisting of a unary
  predicate $P_b$ for each element $b$ of $B$ as well as the complete
  theory $T$ of the structure whose universe consists of the atoms of
  B such that the each predicate $P_b$ is interpreted as the
  collection of atoms in $B$ contained in the element $b$ of $B$. It
  is easy to see that $T$ eliminates quantifier (for $T$ is a complete
  relational monadic theory) and furthermore that $P_b=P_{b_1}$ if and
  only if $b=b_1$ (for $B$ is atomic). An easy induction over the
  complexity of the formula yields that every definable subset of a
  model $m$ of $T$ is a disjunction of subsets of the form \[ \{
  (a_1,\ldots, a_n) \in M^n \ | \ \phi(a_1,\ldots, a_n) \land
  \bigwedge\limits_{i=1}^n P_{b_i}(a_i) \},\] for some $b_1,\ldots,
  b_n$ in $B$ and $\phi$ a quantifier-free formula in the empty
  language. In particular, the theory $T$ is equational and has Morley
  rank $3$, yet it is not Noetherian.

  Whilst it is conceivable that a totally transcendental equational
  complete theory in a countable language or of Morley rank $\le 2$
  has to be Noetherian, we have not further pursued this direction and
  leave the question open.
\end{remark}

\begin{cor}\label{L:deg_formel_type}
  If $M$ is a model, then $\deg_\FF\theta(x,a)$ is the number of types
  $p$ over $M$ containing $\theta(x,a)$ with $\RF(p)=\RF\theta(x,a)$.
\end{cor}
\begin{proof}
  Let $\alpha$ be the $\RF$-rank of $\theta(x,a)$. By Lemma
  \ref{L:Rang_des_Randes}, a type over $M$ of rank $\alpha$ contains
  $\theta(x,y)$ if and only it contains the topological closure
  $\overline{\theta(x,a)}$. Now, the types over $M$ of rank $\alpha$
  containing $\overline{\theta(x,a)}$ correspond exactly to the
  irreducible components of $\overline{\theta(x,a)}$ of rank $\alpha$.
\end{proof}

\begin{lemma}
  Let $M$ be a model of the Noetherian theory $T$ and $p$ a type over
  a subset $A$ of $M$. The minimal formula of $p$ isolates it among
  all types over $A$ of rank at least $\RF(p)$.
\end{lemma}
\begin{proof}
  Let $\phi(x,a)$ be the minimal formula of $p$. Choose another type
  $q\ne p$ over $A$ containing $\phi(x,a)$. There is a formula
  $\theta(x,b)$ in $p$ which implies $\phi(x,a)$ and does not belong
  to $q$. Now, both $\phi(x,a)$ and $\theta(x,b)$ have the same rank
  and degree as $p$, so by Lemma \ref{L:deg_Sum} the rank of
  $\phi(x,a)\land\neg\theta(x,b)$, and therefore also the rank of $q$,
  is strictly smaller than $\RF(p)$, as desired.
\end{proof}

Total transcendence means that Morley rank is ordinal-valued. For the
rest of the section, we will compare Morley rank to the foundational
rank $\RF$ and show equality of these ranks under some mild assumption
(see Definition \ref{D:isol}) on the Noetherian theory $T$.

\begin{cor}\label{C:RM_RF}
  Assume that $T$ is Noetherian with respect to $\FF$. Then, for every
  formula $\theta(x,a)$ with parameters in a model of $T$, we have
  that $\RM\theta(x,a)\le \RF\theta(x,a)$.
\end{cor}
Note that both ranks are computed in reference to the ambient model
under consideration.
\begin{proof}
  By Lemma \ref{L:rk_formel_type}, it is enough to show that
  $\RM(p)\le \RF(p)$ for every type over an $\aleph_0$-saturated model
  $M$. Let $\alpha=\RF(p)$ and $\phi(x,a)$ the minimal formula of $p$.
  Then $\RF(q)<\alpha$ for all types $q\ne p$ containing $\phi(x,a)$.
  By induction on $\alpha$, we deduce that $\RM(q)<\alpha$ for all
  such $q$. Since $M$ is $\aleph_0$-saturated, it follows that
  $\RM(p)\le\alpha$.
\end{proof}

\begin{remark}
  Even for Noetherian theories of finite Morley rank, we need not
  always have equality between Morley rank and the foundational rank
  $\RF$. Indeed, consider the language $\LL$ consisting of a single
  unary predicate $P$ and the theory $T$ whose models are exactly the
  $\LL$-structures where $P$ denotes an infinite co-infinite subset.
  The theory $T$ is Noetherian with respect to the class $\FF$
  consisting of finite conjunctions of atomic formulas. However, the
  irreducible formula $x\doteq x$ has Morley rank $1$ (and Morley degree $2$),
  yet $\RF$-rank $2$.
\end{remark}

One of the reasons why equality of both ranks does not hold in the
above example is the fact that the unique non-algebraic $1$-type over
a model $M$ determined by the formula $\neg P(x)$ contains no
irreducible formula isolating it among all types over $M$ of Morley
rank at least $1$. We will therefore introduce the notion of
Noetherian isolation to ensure equality in Corollary \ref{C:RM_RF}.

\begin{definition}\label{D:isol}
  The Noetherian theory $T$ with respect to $\FF$ admits
  \emph{Noetherian isolation} if every type $p$ over a subset $A$ of a
  model of $T$ contains a closed formula $\phi(x, a)$ which isolates
  $p$ among all types over $A$ of Morley rank at least $\RM(p)$.
\end{definition}

\noindent Clearly $T$ admits Noetherian isolation if and only if the
minimal formula of $p$ isolates it among all types of Morley rank at
least $\RM(p)$.

\begin{theorem}\label{T:isol_equiv}
  The following are
  equivalent for a Noetherian theory $T$:
  \begin{enumerate}[(a)]
  \item\label{T:isol_equiv:noeth} The theory $T$ has Noetherian
    isolation.
  \item\label{T:isol_equiv:gleich} For every formula $\theta(x, a)$
    with parameters in a model of $T$, we have that $\RM\theta(x,
    a)=\RF\theta(x, a)$.
  \item\label{T:isol_equiv:rand} For every consistent formula
    $\theta(x,a)$ with parameters in some model $T$, we have that
    $\RM\bigl(\overline{\theta(x,a)}\land\neg\theta(x,a
    )\bigr)<\RM\theta(x,a)$
  \end{enumerate}
\end{theorem}
\begin{proof}
 For
 (\ref{T:isol_equiv:noeth})$\Rightarrow$(\ref{T:isol_equiv:gleich}):
 By Lemma \ref{L:rk_formel_type} and Corollary \ref{C:RM_RF}, it is
 enough to show that $\RM(p)\ge \RF(p)$ for all types over an
 $\aleph_0$-saturated model $M$. We proceed by induction on
 $\alpha=\RM(p)$. By assumption, the minimal formula $\phi(x,a)$ of
 $p$ isolates $p$ among all types over $M$ of Morley rank at least
 $\alpha$, so $\RM(q)<\alpha$ for all types $q\ne p$ containing
 $\phi(x,a)$. By induction, we have that $\RF(q)<\alpha$, so it
 follows that $\RF\psi(x,b)<\alpha$ for all irreducible proper
 subformulas of $\phi(x,a)$, and thus $\RF\phi(x,a)\le\alpha$, as
 desired.

  The implication
  (\ref{T:isol_equiv:gleich})$\Rightarrow$(\ref{T:isol_equiv:rand})
  follows from Lemma \ref{L:Rang_des_Randes}, so we need only show
  (\ref{T:isol_equiv:rand})$\Rightarrow$(\ref{T:isol_equiv:noeth}).
  Consider a type $p(x)$ over $A$ and let $\theta(x, a)$ in $p$
  isolate it among types over $A$ of Morley rank at least $\RM(p)$.
  Since $\RM\bigl(\overline{\theta(x,a)}\land\neg\theta(x,a
  )\bigr)<\RM\theta(x,a)$, we have that $\overline{\theta(x,a)}$ is a
  closed formula which also isolates $p$ among types over $A$ of
  Morley rank at least $\RM(p)$. Hence, the theory $T$ has Noetherian
  isolation, as desired.
\end{proof}

Notice that the above proof yields immediately the following
corollary.
\begin{cor}\label{C:Noether_iso_Model}
  Suppose that every type over a model $M$ of $T$ is isolated by its
  minimal formula among all types over $M$ of Morley rank at least
  $\RM(p)$. Then $T$ has Noetherian isolation.\qed
\end{cor}

\begin{remark}
  It is easy to see that a theory has Noetherian isolation if Morley
  rank and the foundational rank agree on closed formulas.
\end{remark}

\begin{cor}
  If the Noetherian theory $T$ has Noetherian isolation, then Morley
  degree of a formula $\theta(x, a)$ equals $\deg_\FF\theta(x, a)$.
\end{cor}

\begin{proof}
  By Lemma \ref{L:Rang_des_Randes} and Theorem \ref{T:isol_equiv} part
  (\ref{T:isol_equiv:rand}), we may assume that $\theta(x,a)$ is a
  closed formula. Furthermore, we may also assume that our ambient
  model $M$ is $\aleph_0$-saturated. Now, Morley degree of
  $\theta(x,a)$ is the number of types over $M$ containing
  $\theta(x,a)$ of Morley rank $\RM\theta(x,a)$. Since Morley rank and
  the foundational rank are the same, we have that Morley degree is
  exactly $\deg_\FF\theta(x,a)$, by Corollary \ref{L:deg_formel_type}.
\end{proof}

\begin{cor}
  If the Noetherian theory $T$ has Noetherian isolation, then for
  every type $p$ over a set $A$ with minimal formula $\phi(x, a)$ we
  have that $\RM(p)=\RF\phi(x, a)$ and its Morley degree is
  $\deg_\FF\phi(x,a)$.\qed
\end{cor}

We will conclude this section with an easy observation regarding
imaginaries in Noetherian theories (\cf \cite[Corollary 2.8]{aP07}).

\begin{remark}
  Given a complete Noetherian theory $T$ with respect to the
  Noetherian family $\FF$ of formulae, the theory $T$ has weak
  elimination of imaginaries after adding sorts for the canonical
  parameter of instances of formulae in $\FF$.

  Moreover, if $\FF$ is closed under finite disjunctions, then every
  imaginary is interdefinable with the canonical parameter of some
  instance of a formula in $\FF$.
\end{remark}

\begin{proof}
  Consider an $\emptyset$-definable equivalence relation $E$ and an
  equivalence class $E(x, a)$. By Remark \ref{R:irred}, the closure of
  the constructible subset $E(x, a)$ can be written as an irredundant
  union of its irreducible components $C_1, \dotsc, C_n$. Write each
  $C_i$ as $C_i=\phi_i(x, b_i)$ for some formula $\phi_i$ in $\FF$.
  Clearly, the tuple of canonical parameters $\bar \eta=(\cp{\phi_1(x,
    b_1)}, \dotsc, \cp{\phi_n(x, b_n)})$ is algebraic over the
  canonical parameter of the closure, and thus algebraic over
  $\cp{E(x, a)}$.

  Now, the canonical parameter of the closure is clearly definable
  over the tuple $\bar\eta$, so we need only show that $\cp{E(x, a)}$
  is definable over the canonical parameter of its closure. Otherwise,
  there would be an automorphism $\sigma$ such that constructible set
  $E(x, \sigma(a))$ differs (and thus is disjoint) from $E(x, a)$, but
  they have the same closure, which is not possible since $E(x, a)$
  and $E(x, \sigma(a))$ are constructible.

  If $\FF$ is now closed under finite disjunctions, then the closure
  of $E(x, a)$ is given by a single instance $\phi(x, c)$ of a formula
  $\phi(x,y)$ in $\FF$, so $\cp{E(x, a)}$ and $\cp{\phi(x, c)}$ are
  interdefinable, as desired.
\end{proof}

\section{Proper pairs of algebraically closed fields}\label{S:belles}

As mentioned in the introduction, we will show in Sections
\ref{S:noether_trick} and \ref{S:Hilbertpol} that the theory of proper
pairs of algebraically closed fields is Noetherian. Most of the
results mentioned here already appear in \cite{jK64, bP83}, unless
explicitly stated.

A \emph{pair $(K, E)$ of algebraically closed fields} is an extension
$E\subset K$ of algebraically closed fields. Every pair is a structure
in the language $\LL_P=\LL_{ring}\cup\{P\}$ with $E=P(K)$. If $E=K$,
the pair is bi-interdefinable with the theory $\mathrm{ACF}$ of
algebraically closed fields, which is clearly Noetherian, since
definable sets are Zariski constructible.

\medskip
{\bf Henceforth, we will restrict from now on our attention to proper
  pairs $(K, E)$ of algebraically closed fields, so $E\subsetneq K$.}

\medskip
We denote by $\acfp$ the $\LL_P$-theory of proper pairs of
algebraically closed fields, which expands the incomplete
$\LL_{ring}$-theory $\mathrm{ACF}$ of algebraically closed fields. We
will use the index $P$ to refer to the theory $\acfp$. In particular,
given a subset of parameters $A$ of $K$,  we
mean by $\dcl(A)$ and $\acl(A)$ its definable and algebraic closures in the pure field language,
whereas $\dcl_P(A)$ or $\acl_P(A)$ mean the corresponding objects in
the structure of the pair $(K, E)$. In particular, the independence
symbol $\ind$ refers to the algebraic independence in the reduct
$\mathrm{ACF}$.

As shown by Robinson \cite{aR59}, every completion of the theory
$\acfp$ is obtained by fixing the characteristic. Each of these
completions is $\omega$-stable of Morley rank $\omega$
\cite[p.\ 1659]{bP01}. The induced structure on the proper subfield
$E$ agrees with its structure as a pure field, so $E$ has Morley rank
$1$. Over any subset of parameters $A$ of $K$, there is a unique type
of Morley rank $\omega$ over $A$ given by an element (inside a
sufficiently saturated pair) which is transcendental over the
compositum field $E\cdot \mathrm{Quot}(A)$, where $ \mathrm{Quot}(A)$
denotes the subfield generated by $A$.

Delon \cite{fD12} provided a definable expansion of the language
$\LL_P$ for $\acfp$ to have quantifier elimination. The suggested
language is $\LLD=\LL_P\cup\{\dep_n, \lambda_n^i\}_{1\le i\le n\in
  \N}$, where $K\models \dep_n(a_1,\dotsc,a_n)$ if and only if
$a_1,\dotsc,a_n$ are $E$-linearly independent. Each function
$\lambda_n^i$ takes values in $E$. If $a_1,\dotsc,a_n$ are
$E$-linearly independent, but $a_0,a_1,\dotsc,a_n$ are not, its values
are determined by $a_0 = \sum\limits_{i=1}^n \lambda_n^i(a_0;
a_1\dotsc,a_n)\,a_i$. Otherwise, the value $\lambda_n^i(a_0;
a_1\dotsc,a_n)$ is zero.

\begin{notation}
  From now on, given a set $A$ of parameters, we will denote by
  $\lambda(A)$ the subfield of $E$ generated by the values of the
  $\lambda$-functions applied to tuples of $A$.
\end{notation}
Note that the subfield $E\cap k$ is always contained in $\lambda(k)$,
since an element $a$ belongs to $E$ if and only if $a=\lambda(a;1)$.
\begin{remark}\label{R:lambda_min}
  For every subfield $k$ of $K$, we have that $\lambda(k)$ is the
  smallest subfield $F$ of $E$ such that $F\cdot k$ and $E$ are
  linearly disjoint over $F$. This implies
  $\lambda(k(e))=\lambda(k)(e)$ for every tuple $e$ in $E$. In
  particular, \[ \lambda(k\cdot\lambda(k))=\lambda(k).\]
\end{remark}

\begin{lemma}\label{L:indep_lambda}
  Let $k\subset L$ be subfields of $K$. Then,
  \[L\ind_{\lambda(k)\cdot k}E\;\;\Leftrightarrow\;\;
  \lambda(L)\subset\acl\bigl(\lambda(k)\bigr).\]
\end{lemma}

\begin{proof}
  By Remark \ref{R:lambda_min}, we have $\lambda(k)\cdot
  k\ind_{\lambda(k)}E$. Now, transitivity and monotonicity of
  non-forking yield that
  \[L\ind_{\lambda(k)\cdot
  k}E\;\;\Leftrightarrow\;\;L\ind_{\lambda(k)}E.\] The latter
  condition is equivalent to $\acl(\lambda(k))\cdot
  L\ld_{\acl(\lambda(k))}E$, which again by Remark \ref{R:lambda_min}
  is equivalent to $\lambda(L)\subset\acl(\lambda(k))$.
\end{proof}

\begin{definition}
  A subfield $k$ of $K$ is \emph{\lc} if $\lambda(k)$ is a subfield of
  $k$, or equivalently, if $k$ is linearly disjoint from $E$ over the
  subfield $E\cap k$, which again is equivalent to $E\cap
  k=\lambda(k)$.
\end{definition}

A straightforward application of Remark \ref{R:lambda_min} and Lemma
\ref{L:indep_lambda} to the subfields $ k(e)\subset k(a,e)$ yields the
following result:
\begin{cor}\label{C:ind_lambda}
  If $k$ is $\lambda$-closed, then for all tuples $a$ in $K$ and $e$
  in $E$, we have
  \[a\ind_{k(e)}E\;\;\Leftrightarrow\;\;
  \lambda(k(a))\subset\acl\bigl(\lambda(k)(e)\bigr)\] \qed
\end{cor}

\begin{fact}\label{F:Delon}\textup{(}\cite[Th\'eor\`eme 1]{fD12}\textup{)}~
  The fraction field of an $\LLD$-substructure is \lc. The
  $\LL_P$-type of a \lc field (seen as a long tuple with respect to
  some fixed enumeration) is uniquely determined by its
  quantifier-free $\LL_P$-type, so the theory $\acfp$ has quantifier
  elimination in the language $\LLD$.
\end{fact}

\begin{remark}\label{R:lambda_closed}
  Every $\LL_P$-definably closed subset of a model $(K, E)$ of $\acfp$
  is \lc as a subfield of $K$. Moreover, if a subfield $k$ is \lc,
  then its $\LL_P$-definable closure is its inseparable closure
  $k^\mathrm{ins}$ and its $\LL_P$-algebraic closure is the field
  algebraic closure $k^\mathrm{alg}$.
\end{remark}

We now introduce (cf. \cite[Definition 6.3]{MPZ20}) the collection of
\emph{tame formulae}, which will be shown in Sections
\ref{S:noether_trick} and \ref{S:Hilbertpol} to be Noetherian.

\begin{definition}\label{D:tame}
  A \emph{tame formula} on the tuple $x$ of variables is an
  $\LL_P$-formula of the form
  \[\exists\, \zeta \in P^r \biggl(\neg\,\zeta\doteq 0\;
  \land\; \Land\limits_{j=1}^{m} q_j(x,\zeta)\doteq 0 \biggr)\] for
  some polynomials $q_1,\dotsc, q_m$ in $\Z[X,Z]$, homogeneous in the
  variables $Z$.
\end{definition}

\begin{fact}\label{F:tame1}\textup{(}\cite[Lemma 6.4, Corollaries
    6.5 and 6.8 \& Proposition 7.3]{MPZ20}\textup{)}~
  \begin{itemize}
  \item Given polynomials $q_1,\dotsc, q_m$ in $\Z[X,Y,Z]$ homogeneous
    in the variables $Y$ and $Z$ separately, the $\LL_P$-formula
    \[\exists\, \xi \in P^r\; \exists\,\zeta \in
    P^s\Bigl(\neg\,\xi\doteq 0\; \land\; \neg\,\zeta\doteq
    0\;\land\;\Land_{j=1}^mq_j(x,\xi,\zeta)\doteq 0 \Bigr)\] is
    equivalent in $\acfp$ to a tame formula.
  \item The collection of tame formulae is, up to equivalence, closed
    under finite conjunctions and disjunctions.
  \item Every tame formula, in any partition of the variables, is an
    equation (cf. Remark \ref{R:noeth_eq}).
  \item Every $\LL_P$-formula is equivalent modulo $\acfp$ to a
    Boolean combination of tame formulae, so $\acfp$ is equational.
  \end{itemize}
\end{fact}

The fundamental reason why tame formulae are equations is due to the
following observation, which will be again relevant in order to show
that $\acfp$ is Noetherian:

\begin{remark}
  Projective varieties are \emph{complete}: Given a projective variety
  $Z$ and an algebraic variety $X$, the projection map $X \times Z\to
  X$ is closed with respect to the Zariski topology.
\end{remark}

By Fact \ref{F:tame1}, in order to show that $\acfp$ is Noetherian, we
need only show in Sections \ref{S:noether_trick} and
\ref{S:Hilbertpol} that the family of instances of tame formulae,
which is already closed under finite intersections, has the DCC. For
this, we need a couple of auxiliary lemmata. The next result already
appeared implicitly in \cite[Lemma 7.2]{MPZ20} (or Lemma 9.1 in the
ArXiv version), so we will not provide a proof thereof.

\begin{lemma}\label{L:tame=Zar}
  Let $k$ be a \lc subfield of $K$. For every instance $\phi(x, a)$ of
  a tame formula $\phi$ with parameters in $k$, there exists a Zariski
  closed subset $V$ of $E^{|x|}$ defined over $\lambda(k)$ such that
  for every $e$ in $E$, \[ (K, E)\models \phi(e, a)
  \ \Longleftrightarrow \ e\in V.\] If the polynomials in $\phi$ are
  homogeneous in $X$, then so are the polynomials defining $V$.\qed
\end{lemma}

We will finish this section by showing (\cf Remark \ref{R:minimal_A})
that every \emph{closed} formula (as in Definition \ref{D:noether_F}
with $\FF$ the family of tame formulae) over a subset $A$ of
parameters is indeed equivalent to an instance of a tame formula with
parameters over $A$. This result resonates with \cite[Proposition
  2.9]{Ju00}.

\begin{prop}\label{P:tame_instance}
  An instance of a tame formula which is equivalent to an
  $\LL_P$-formula with parameters in $A$ is equivalent to an instance
  of a tame formula with parameters over $A$.
\end{prop}

\begin{proof}
  To render the presentation of the proof more structured, we will
  first prove some intermediate claims.

  \begin{claim}\label{Claim:dcl}
    Every instance of a tame formula with parameters in $\dcl_P(A)$ is
    equivalent to an instance of a tame formula over $A$.
  \end{claim}

  \begin{claimproof}
    By Remark \ref{R:lambda_closed}, the definable closure $\dcl_P(A)$
    is the inseparable closure of the smallest \lc subfield containing
    $A$. Thus, the parameters from $\dcl_P(A)$ are obtained from $A$
    using the ring operations, as well as inversion, extraction of
    $p^{th}$-roots (if the characteristic of $K$ is $p>0$) and
    applying the $\lambda$-functions. The cases of the ring
    operations, inversion and extraction of $p^{th}$-roots are easy
    (for the distinguished algebraically closed subfield $E$ is
    perfect), so we will solely focus on the application of the
    $\lambda$-functions. For the sake of the presentation, assume that
    $a_0, \ldots, a_n$ are elements of $A$ with
    $e_1=\lambda_n^1(a_0;a_1,\ldots, a_n)\ne 0$ (so $a_1,\ldots, a_n$
    are linearly independent over $E$). Consider now the instance
    \[ \phi(x, a', e_1)= \exists\, \zeta \in P^r \biggl(\neg\,\zeta\doteq 0\;
    \land\; \Land\limits_{j=1}^{m} q_j(x, a', e_1,\zeta )\doteq 0
    \biggr)\] of a tame formula, where $a'$ is a tuple in $A$
    containing $a_0,\ldots, a_n$. Let $N$ be the largest integer such
    that $e_1^N$ occurs non-trivially in some $q_j$. Set
    now \begin{multline*} \psi(x, a')= \exists\, \zeta \in P^r
      \ \exists\, \xi \in P^{n+1} \biggl(\neg\,\zeta\doteq 0\; \land\;
      \neg\,\xi \doteq 0\; \land \; \\ \xi_0 a_0=\sum\limits_{i=1}^n
      \xi_i a_i\; \land \;\Land\limits_{j=1}^{m} \xi_0^N q_j(x, a',
      \frac{\xi_1}{\xi_0},\zeta )\doteq 0 \biggr),
    \end{multline*}
    which is an instance of a tame formula over $A$ by Fact
    \ref{F:tame1}. Observe that the element $\xi_0$ in $\psi(x, a')$
    is never $0$, for $a_1,\ldots, a_n$ are linearly independent over
    $E$, so $e_1=\frac{\xi_1}{\xi_0}$. Thus, the two instances are
    equivalent, as desired.
  \end{claimproof}
  \begin{claim}\label{Claim:acl}
    An instance of a tame formula with parameters in $\acl_P(A)$ which
    is equivalent to an $\LL_P$-formula with parameters in $A$ is
    equivalent to an instance of a tame formula over $A$.
  \end{claim}

  \begin{claimproof}
    Suppose the element $b$ is algebraic over $A$ and consider the
    instance
    \[\phi(x,b)=\exists\, \zeta \in P^r \biggl(\neg\,\zeta\doteq 0\;
    \land\; \Land\limits_{j=1}^{m} q_j(x,b,\zeta)\doteq 0 \biggr),\] or
    equivalently using a different notation
    \[\phi(x,b) = \exists\, \zeta \in P^r \Bigl(\neg\,\zeta\doteq 0\;
    \land\; (x,\zeta)\in\\V\bigl(I(b)\bigr)\Bigr),\] where $I(b)$ is
    the ideal of $K[X,Z]$, homogeneous in $Z$, generated by
    $q_1(X,b,Z)$,\ldots, $q_n(X,b,Z)$.

    Let now $b=b_1,\ldots, b_n$ be the $\LL_P$-conjugates of $b$ over
    $A$. Since $\phi(x, b)$ is equivalent to an $\LL_P$-formula with
    parameters in $A$, we have that $\phi(x, b)$ is equivalent to the
    disjunction $\bigvee_{i=1}^n \phi(x, b_i)$, which is again an
    instance of a tame formula, namely
    \[\exists\, \zeta \in P^r \bigl(\neg\,\zeta\doteq 0\;
    \land\; (x,\zeta)\in\V(J)\bigr),\] where $J$ is the product ideal
    $I(b_1)\cdots I(b_n)$. The ideal $J$ is invariant under all
    automorphisms of $(K, E)$ fixing $A$ pointwise, so by Weil's
    theorem \cite{aW56} its field of definition is contained in $\dcl_P(A)$.
    Hence, the ideal $J$ can be generated by polynomials over
    $\dcl_P(A)$ which are homogeneous in $Z$. Therefore, the instance
    $\phi(x, b)$ is equivalent to an instance of a tame formula with
    parameters in $\dcl_P(A)$. By Claim \ref{Claim:dcl}, we conclude
    that $\phi(x, b)$ is equivalent to an instance of a tame formula
    with parameters over $A$, as desired.
  \end{claimproof}

  We now have all the ingredients to prove the statement of the
  proposition. Consider thus an instance $\phi(x, b)$ of a tame
  formula and assume that $\phi(x, b)$ is equivalent to an
  $\LL_p$-formula $\theta(x,a)$ with parameters over $A$. Consider
  first the case that $A$ is not fully contained in $E$, so by Fact
  \ref{F:Delon} and Remark \ref{R:lambda_closed}, the subset
  $\acl_P(A)$ is the universe of an elementary substructure $k$ of
  $(K, E)$. Hence, there is some $b'$ in $\acl_P(A)$ such that
  $\phi(x, b')$ is equivalent to $\theta(x, a)$ (and thus to $\phi(x,
  b)$). We deduce from Claim \ref{Claim:acl} that $\phi(x, b)$ is
  equivalent to an instance of a tame formula over $A$, as desired.

  Thus, we need only consider the case that the parameter set $A$ is a
  subset of $E$. Choose a small elementary substructure $k$ of $(K,E)$
  containing $A$. By saturation, there is some element $a'$ in $K$
  which is transcendental over the subfield $E\cdot k$. Set now
  $A'=A\cup \{a'\}$ and deduce from the above paragraph as well as
  from Claim \ref{Claim:acl} that $\phi(x, b)$ is equivalent to an
  instance \[\phi_1(x,a, a')=\exists\, \zeta \in P^r
  \biggl(\neg\,\zeta\doteq 0\; \land\; \Land\limits_{j=1}^{m}
  q_j(x,a,a',\zeta)\doteq 0 \biggr),\] where $a$ is a tuple of
  elements in $A$. Let $c$ in $k$ be a realisation of $\phi_1(x, a,
  a')$ and $e$ some tuple in $E$ . Since $a'$ is transcendental over
  $E\cdot k$, we have that $q_j(c, a, a',e)=0$ if and only if the
  polynomial $q_j(c, a, Y, e)$ is the trivial polynomial (which is
  equivalent to a finite system of polynomial equations). Set now
  \[\phi_2(x,a)=\exists\, \zeta \in P^r \biggl(\neg\,\zeta\doteq 0\;
  \land\; \Land\limits_{j=1}^{m} q_j(x,a,Y',\zeta)\doteq 0 \biggr),\]
  which is again an instance of a tame formula with parameters in $A$.
  It is now clear that \[ \theta(k, a)=\phi(k, b)=\phi_1(k, a,
  a')=\phi_2(k, a).\] Since $\theta(x, a)$ and $\phi_1(x, a)$ have
  parameters in $A\subset k$, we conclude that $\phi_2(x, a)$ is
  equivalent to $\theta(x, a)$, and thus to $\phi(x, b)$, as desired.
\end{proof}

\section{An indirect proof of the Noetherianity of $\acfp$}
\label{S:noether_trick}

In this section we will give a simple proof that the instances of tame
formulae have the DCC in the theory $\acfp$ of proper pairs of
algebraically closed fields. Whilst the methods we will use for the
proof are elementary, using the strength of Corollary \ref{C:min},
they do not explicitly allow to produce the minimal tame formula in a
given type. The subsequent Section \ref{S:Hilbertpol} will provide an
explicit description of the minimal tame formula of a given type using
results on the Hilbert polynomials of saturated ideals.

\begin{prop}\label{P:Trick_proof}
  Given a \lc subfield $k$ and a finite tuple $a$ of $K$, there is
  some $\LL_P$-formula $\theta(x)$ in $\tp_P(a/k)$ which implies every
  instance of a tame formula in $\tp_P(a/k)$.
\end{prop}

\begin{proof}
  With respect to a fixed compatible total order of the collection of
  monomials on $X$, choose a Gr\"obner basis $r_1,\dotsc, r_m$ of the
  vanishing ideal $\I(a/ E\cdot k)$, that is, a collection of
  polynomials $r_1,\dotsc, r_m$ in the ideal $\I(a/ E\cdot k)$ such
  that the leading (or \emph{initial}) monomial (with respect to the
  fixed total order) of every polynomial in the ideal is divisible by
  the leading monomial of one of the $r_j$'s. Such a set is always a
  generating set for the ideal thanks to the corresponding division
  algorithm of Gr\"obner \cite[Proposition 5.4.2 \& Corollary
    5.4.5]{nL08}.

  Clearing denominators, we may assume that each $r_i$ has
  coefficients in the ring $k[E]$ generated by $k\cup E$. Since every
  element of the ring $k[E]$ is a sum of products of the form $b\cdot
  e'$, where $b$ is in $k$ and $e'$ is an element of $E$, we may write
  (after possibly adding additional variables) each $r_i$ as
  $r_i=r_i(X,e)$, where $r_i(X,Z)$ is a polynomial in $k[X,Z]$ and $e$
  is a tuple from $E$. Let $\sigma_i(e)$ be the leading coefficient of
  $r_i(x, e)$ with respect to our compatible total order and denote by
  $\sigma(e)$ the product of all the $\sigma_i(e)$'s. Choosing a
  system of generators of the vanishing ideal $I(e/k)$, denote by
  $\gamma(Z)$ the locus of $e$ over $k\cap E=\lambda(k)$.

  We will show that the $\LL_P$-formula \[\theta(x) = \exists\,
  \zeta\in P\;\bigl( \gamma(\zeta)\land\neg\,\sigma(\zeta)\doteq
  0\land \Land_{i=1}^l\,r_i(x,\zeta)\doteq 0\bigr)\] has the desired
  property as in the statement. Notice first of all that the above
  formula belongs to $\tp_P(a/k)$, setting $\zeta=e$.

  Consider now an instance of a tame formula
  \[\phi(x)=\exists\;\zeta'\in P\;\bigl(\neg\,\zeta'\doteq
  0\;\land\;\Land_{j=1}^m\,q_j(x,\zeta')\doteq 0\bigr)\] in
  $\tp_P(a/k)$, where the polynomials $q_j$ in $k[X,Z']$ are
  homogeneous in the tuple of variables $Z'$. Since $a$ realises
  $\phi$, there exists a non-trivial tuple $e'$ in $E$ with $q_j(a,
  e')=0$ for every $1\le j\le m$. Now, each polynomial $q_j(X,e')$
  belongs to $\I(a/ E\cdot k)$, so for large enough $N$ and each $j$
  Gr\"obner's division algorithm allows us to write
  \[\sigma(e)^Nq_j(X,e')=\sum_{i=1}^l h_{j,i}(X,e,e')r_i(X,e)\]
  for some polynomials $h_{j,i}(X,Z,Z')$ over $k$, homogeneous in
  the tuple of variables $Z'$ of the same degree as $q_j(X,Z')$.

  The formula
  \[\rho(\zeta)=\exists\;\zeta'\in P\;\Bigl(\neg\,\zeta'\doteq
  0\;\land\;\Land_{j=1}^m\bigl(\sigma(\zeta)^Nq_j(X,\zeta')\; -\!
  \sum_{i=1}^l h_{l,i}(X,\zeta,\zeta')r_i(X,\zeta) \doteq
  0\bigr)\Bigr)\] is an instance of a tame formula in the type
  $\tp_P(e/k)$. Since $k$ is \lc, Lemma \ref{L:tame=Zar} yields that
  $\rho(E)$ is equivalent to the $E$-rational points of a Zariski
  closed set defined over $\lambda(k)=k\cap E$. In particular, every
  solution of the locus $\gamma(\zeta)$ of $e$ over $k$ must satisfy
  $\rho(\zeta)$.

  Choose now a realisation $b$ of the above formula $\theta(x)$ and
  let $f$ be the corresponding tuple from $E$. Since $f$ is a solution
  of $\gamma(\zeta)$, it realises $\rho(\zeta)$, so there exists a
  non-trivial tuple $f'$ in $E$ such that
  $\sigma(f)^Nq_j(X,f')=\sum_{i=1}^l h_{j,i}(X,f,f')r_i(X,f)$ for
  every $1\le j\le m$. Since $r_i(b,f)=0$ for all $i$ yet
  $\sigma(f)\ne 0$, it follows that $q_j(b,f')=0$ for all $1\le j\le
  m$. We conclude that every realisation $b$ of $\theta$ realises
  $\phi$, as desired.
\end{proof}

Corollary \ref{C:min}, Fact \ref{F:tame1} and Proposition
\ref{P:Trick_proof} immediately yield the following.

\begin{cor}\label{C:acfp_noether}
  The theory $\acfp$ of proper pairs of algebraically closed fields is
  Noetherian with respect to the collection of tame formulae.\qed
\end{cor}
By Remark \ref{R:minimal_A} and Proposition \ref{P:tame_instance}, we
deduce the following result.
\begin{cor}
  Every type over a subset $A$ of $K$ contains an instance $\phi(x,a)$
  of a tame formula which implies every formula in $p$ which is
  equivalent to an instance of a tame formula. \qed
\end{cor}

\section{Morley, Lascar and Poizat}\label{S:RM}

In \cite[Subsection 2.2, p.\ 1660]{bP01}, Poizat stated (without
proof) that the following rank equality holds for every type
$p=\tp(a/k)$ over an elementary substructure $k$ of a sufficiently
saturated proper pair $(K, E)$ of algebraically closed fields: \[
\RU(p)=\RM(p)=\omega \cdot \tr(a/E\cdot k) + \tr(\alpha/E\cap k),\]
where $\alpha$ is the canonical base in the reduct ACF of $k(a)$ over
$E$. (Note that there is a misprint in \cite{bP01}). He deduced from
the above identity that the dimension associated to the unique generic
type of $K$ is additive. Though Poizat's formula (and its proof) is
probably well-known, we will nonetheless take the opportunity to give
a detailed proof in this section. Our proof yields in particular that
the Morley rank of a type over a \lc subfield can be isolated by a
tame formula, and thus the theory $\acfp$ admits Noetherian isolation,
by Corollary \ref{C:Noether_iso_Model}. In order to prove Poizat's
formula, we will need some auxiliary results regarding the behaviour
of non-forking independence in $\acfp$.

The theory $\acfp$ of proper pairs of algebraically closed fields is a
particular case of a more general construction due to Poizat
\cite{bP83}, who showed that the common theory of \emph{belles paires}
of models of a stable theory $T$ is again stable whenever $T$
\emph{does not have the finite cover property ($\mathrm{nfcp}$)}. For
a stable theory, nfcp is equivalent \cite[Chapter II, Theorem
  4.4]{Shelahbook} to the elimination of $\exists^\infty$ in
$T^{\mathrm eq}$. The theory of algebraically closed fields eliminates
both imaginaries and the quantifier $\exists^\infty$, so it has nfcp.
However, there are Noetherian theories with fcp, as the following
example shows.

\begin{remark}
  In the language $\LL$ consisting of a single binary relation
  $E(x,y)$ for an equivalence relation, consider the theory of the
  $\LL$-structures which have exactly one equivalence class of size
  $n$ for every $1\le n$ in $\N$. This theory is $\omega$-stable of
  Morley rank $2$ and admits quantifier elimination after adding
  countably many constant symbols as canonical representatives of the
  finite equivalence classes. In particular, this theory is
  Noetherian, but does not eliminate $\exists^\infty$, witnessed by
  the formula $E(x,y)$.
\end{remark}

In \cite{BPV03}, Ben Yaacov, Pillay and Vassiliev generalised Poizat's
\emph{belles paires} of stable structures to pairs of models of a
simple theory. Akin to the result of Poizat, when the corresponding
theory of pairs is first-order, then it is again simple. Moreover,
non-forking independence in the theory of the pair (which we will
denote by $\ind^P$) can be characterised in terms of the
independence(s) in the $\LL$-reduct of the sets as well as of the
canonical bases over the predicate, which in our setting corresponds
to taking $\lambda$-closures, up to interalgebraicity.

All throughout this section, work inside a sufficiently saturated
proper pair $(K, E)$ of the theory $\acfp$ of proper pairs of
algebraically closed fields. All subsets and tuples considered are
small with respect to the saturation of $(K,E)$.
\begin{fact}\label{F:forking}\textup{(}\cite[Remark 7.2 \&
    Proposition 7.3]{BPV03}\textup{)}~ Let $a$ be a finite tuple and
  $k\subset L$ be subfields of $K$. We have the following description
  of non-forking:

  \[ \begin{array}{rcl}a\ind^P\limits_{k} L & \mbox{ if and only if }
    & \left\{\begin{array}{c} k(a)\ind\limits_{ E\cdot k} E\cdot
    L\\[5mm] \mbox{and}\\[3mm] \lambda(k(a)) \ind\limits_{\lambda(k)}
    \lambda(L)
    \end{array}\right.
  \end{array}\]
\end{fact}

From the above description of non-forking, we easily deduce the
following consequence.

\begin{lemma}\label{L:lambda_ind}
  Assume $k\subset L$ are subfields of $K$. Given a tuple $a$ in $K$,
  whenever \[k(a)\ind_{ E\cdot k} E\cdot L, \] we have that
  $\lambda(L(a))$ is interalgebraic with $\lambda(k(a))$ over
  $\lambda(L)$. In particular, if $a\ind^P_k L$, then $\lambda(L(a))$
  and $\lambda(k(a))$ are interalgebraic over $\lambda(L)$.
\end{lemma}

\begin{proof}
  Note first that $\lambda(k(a))$ is contained in $\lambda(L(a))$. By
  Remark \ref{R:lambda_min}, we need only show that
  \[  \big(\lambda(L)\cdot \lambda(k(a))\big)^\mathrm{alg}
  \cdot L(a) \ld_{\big(\lambda(L)\cdot
    \lambda(k(a))\big)^\mathrm{alg}} E, \] or equivalently, \[
  \lambda(k(a)) \cdot L(a) \ind_{\lambda(L)\cdot \lambda(k(a))} E.\]
  Now, the fields $k(a)$ and $E$ are linearly disjoint over
  $\lambda(k(a))$, so \[ \lambda(k(a))\cdot k(a)
  \ind_{\lambda(k(a))\cdot k} E\cdot k.\] Together with the assumption
  \[  k(a)\ind_{E\cdot k} E\cdot L,\] we deduce by transitivity of
  non-forking independence that \[ \lambda(k(a))\cdot k(a)
  \ind_{\lambda(k(a))\cdot k\cdot } E\cdot L,\] and thus \[
  \label{E:lambda_ind}
  \lambda(k(a))\cdot L(a) \ind_{\lambda(k(a))\cdot L} E\cdot
  L.\tag{$\star$}\] Remark \ref{R:lambda_min} yields that the subfield
  $\lambda(L)\cdot L$ is linearly disjoint from $E$ over $\lambda(L)$,
  and therefore by monotonicity \[ \lambda(k(a))\cdot L
  \ind_{\lambda(L)\cdot \lambda(k(a))} E.\] Together with
  $\eqref{E:lambda_ind}$, we conclude by transitivity
  that \[\lambda(k(a))\cdot L(a) \ind_{\lambda(L)\cdot \lambda(k(a))}
  E,\] as desired.
\end{proof}

\begin{notation}
  Given a tuple $a$ and a subfield $k$ of $K$, set \[ \smallrm(a/k)=
  \omega \cdot \tr(a/E\cdot k) + \tr(\lambda(k(a))/\lambda(k)).\]
\end{notation}

Poizat's formula now translates as $ \RU(a/k)=\RM(a/k)=\smallrm(a/k)$.
In order to show that these three ranks agree, we will first show that
the rank $\smallrm$ controls non-forking.

\begin{lemma}\label{L:nonfork_smallrm}
  The ordinal-valued rank $\smallrm$ witnesses non-forking: Given
  subfields $k\subset L$ and a tuple $a$ of $K$, we have that
  \[ a\ind^P_k L \text{ if and only if }
  \smallrm(a/k)=\smallrm(a/L).\]
\end{lemma}

\begin{proof}
  Adding to a set the values of its $\lambda$-functions does not
  affect non-forking independence in $\acfp$, since the
  $\lambda$-functions are $\LL_P$-definable. Moreover, none of the
  transcendence degrees occurring in $\smallrm(a/k)$ change when
  passing from $k$ to the intermediate \lc field extension $k\subset
  \lambda(k)\cdot k \subset E\cdot k$, so we may assume that $k$, and
  analogously $L$, is \lc.

  We prove first that the rank $\smallrm$ remains constant when
  passing to a non-forking extension. By the description of
  non-forking in Fact \ref{F:forking}, we have that
  \[ \label{E:Indep} k(a)\ind_{E\cdot k} E\cdot L\ \text{ and } \
  \lambda(k(a))\ind_{\lambda(k)} \lambda(L) \tag{$\natural$}, \] so
  $\tr(a/E\cdot k)\stackrel{\eqref{E:Indep}}{=}\tr(a/E\cdot L)$. Now,
  Lemma \ref{L:lambda_ind} yields that $\lambda(L(a))$ is
  interalgebraic with $\lambda(k(a))$ over $\lambda(L)$, so \[
  \tr(\lambda(k(a))/\lambda(k))\stackrel{\eqref{E:Indep}}{=}
  \tr(\lambda(k(a))/\lambda(L)) = \tr(\lambda(L(a))/\lambda(L)).\]
  Therefore,
  \begin{multline*}
    \smallrm(a/k)=\omega \cdot \tr(a/E\cdot k) +
    \tr(\lambda(k(a))/\lambda(k)) = \\ = \omega \cdot \tr(a/E\cdot L)
    + \tr(\lambda(L(a))/\lambda(L))=\smallrm(a/L), \end{multline*} as
  desired.

  Let us now prove the converse: If $a\nind^P_k L$, then
  $\smallrm(a/L)<\smallrm(a/k)$. Again by Fact \ref{F:forking}, one of
  the two independences in $\eqref{E:Indep}$ cannot hold. If
  $k(a)\nind_{E\cdot k} E\cdot L$, the leading coefficient of $\omega$
  in $\smallrm(a/L)$ is strictly smaller than the coefficient of
  $\omega$ in $\smallrm(a/k)$, so we are done. We may thus assume that
  $k(a)\ind_{E\cdot k} E\cdot L$ and hence $\lambda(L(a))$ is
  interalgebraic with $\lambda(k(a))$ over $\lambda(L)$ by Lemma
  \ref{L:lambda_ind}. However \[ \lambda(k(a))\nind_{\lambda(k)}
  \lambda(L),\] so $\tr(\lambda(L(a))/\lambda(L)) =
  \tr(\lambda(k(a))/\lambda(L)) < \tr(\lambda(k(a))/\lambda(k))$. We
  conclude that $\smallrm(a/L)<\smallrm(a/k)$, as desired.
\end{proof}

In this section, we will show Poizat's formula in two steps: First, we
show that the rank $\smallrm$ is \emph{connected} (\cf Lemma
\ref{L:U_smallrm}), so it must be bounded from above by Lascar rank,
as both ranks witness non-forking. We will then show that every type
over a \lc subfield can be isolated by a tame formula among types of
larger $\smallrm$-rank, which will then give that the rank $\smallrm$
is bounded from below by Morley rank. Now, the inequality $\RU(p)\le
\RM(p)$ always holds for all types, so putting all together we obtain
the equality of all three ranks.

\begin{lemma}\label{L:U_smallrm}
  Consider a subfield $k$ of $K$ and a finite tuple $a$. Assume that
  $\alpha< \smallrm(a/k)$ for some ordinal number $\alpha$. Then there
  is some field extension $k\subset L$ with $\alpha\le
  \smallrm(a/L)<\smallrm(a/k)$. It follows that $\smallrm(p) \le
  \RU(p)$ for every type $p$, by Lemma \ref{L:nonfork_smallrm}.
\end{lemma}

\begin{proof}
  As in the proof of Lemma \ref{L:nonfork_smallrm}, we may assume that
  $k$ is \lc. The proof follows immediately by transfinite induction
  from the following two claims:
  \begin{claim}
    If $\smallrm(a/k)=\beta+1$, then there is some $L\supset k$ with
    $\smallrm(a/L)=\beta$.
  \end{claim}

  \begin{claimproof}
    Write \[ \beta+1=\smallrm(a/k)=\omega\cdot \tr(a/E\cdot k)+
    \tr(\lambda(k(a))/\lambda(k)),\] so
    $0<\tr(\lambda(k(a))/\lambda(k))=m+1$ for some natural number $m$.
    In particular, there is a transcendental element $e$ in
    $\lambda(k(a))$ over $\lambda(k)$. Set $L=k(e)$. Notice that $
    \lambda(L)=\lambda(k)(e)$ by Remark \ref{R:lambda_min}, since $k$
    and $E$ are linearly disjoint over $\lambda(k)$. Analogously, we
    have that $\lambda(L(a))=\lambda(k(a))(e)$. As $\tr
    (\lambda(k(a))/\lambda(k)(e))=m$, we conclude that \[
    \smallrm(a/L)=\omega\cdot \tr(a/E\cdot L) +
    \tr(\lambda(L(a))/\lambda(L))= \omega\cdot \tr(a/L\cdot E)
    +m=\beta,\] as desired.
  \end{claimproof}

  \begin{claim}
    If $\smallrm(a/k)=\omega\cdot (n+1)$ for some $n$ in $\N$, then
    for every $m$ in $\N$ there is some field extension $k\subset L$
    with $\omega\cdot n+m\le \smallrm(a/L)<\smallrm(a/k)$.
  \end{claim}

  \begin{claimproof}
    Since \[ \omega\cdot (n+1)=\smallrm(a/k)=\omega\cdot \tr(a/E\cdot
    k)+\tr(\lambda(k(a))/\lambda(k)),\] we deduce that $\lambda(k(a))$
    is algebraic over $\lambda(k)$, so $k(a)$ and $E$ are
    algebraically independent over $\lambda(k)$. Moreover, the
    transcendence degree $\tr(a/E\cdot k)$ is strictly positive, so
    choose some element $c$ in $k(a)$ transcendental over $E\cdot k$.
    By saturation, there are elements $e_0, \dotsc, e_{m-1}$ in $E$
    transcendental over $k$. Set $b=\sum_{i=0}^{m-1} e_i c^i$ and
    notice that that $b$ and $c$ are interalgebraic over $E\cdot k$.
    Thus, the element $b$ is transcendental over $E\cdot k$. It
    follows that $\tr(a/E\cdot k(b))=n<\tr(a/E\cdot k)$. Set thus
    $L=k(b)$ and notice that \[ \smallrm(a/L)<\smallrm(a/k).\] The
    field $L$ is trivially linearly disjoint from $E\cdot k$ over $k$,
    since $b$ is transcendental over $E\cdot k$. Therefore, the field
    $\lambda(L)$ equals $\lambda(k)$ by Remark \ref{R:lambda_min} and
    transitivity, for $k$ is \lc.

    Both elements $b$ and $c$ belong to $L(a)$, so $e_0, \dotsc,
    e_{m-1}$ lie in $\lambda(L(a))$. Hence, \[
    \tr(\lambda(L(a))/\lambda(L))\ge \tr(e_0, \dotsc,
    e_{m-1}/\lambda(k))=m.\] Therefore \[\smallrm(a/k) >
    \smallrm(a/L)=\omega\cdot \tr(a/E\cdot L) +
    \tr(\lambda(L(a))/\lambda(L))\ge \omega\cdot n +m,\] as
    desired.\end{claimproof}
\end{proof}

We are now left to bounding the Morley rank from above in terms of the
rank $\smallrm$. We will do so in terms of an explicit tame formula
$\chi$ which will isolate the type $p$ among those types
$\smallrm$-rank at least $\smallrm(p)$. However, the tame formula
$\chi$ we exhibit need not be the minimal tame formula in the type $p$
(as in Corollary \ref{C:min}).

\begin{prop}\label{P:rm_isol_tame}
  Consider a finite tuple $a$ and a \lc subfield $k$ of $K$. There
  exists an instance $\chi$ in $\tp_P(a/k)$ of a tame formula such
  that $\smallrm(b/k)\le \smallrm(a/k)$ for every realisation $b$ of
  $\chi$ in $K$. Moreover,
  \[ \smallrm(b/k)=\smallrm(a/k) \ \Longleftrightarrow \
  \tp_P(b/k)=\tp_P(a/k)\]
\end{prop}

\begin{proof}
  Let $n=\tr(a/E\cdot k)$. Using \cite[Theorem III.8]{sL58} we
  conclude from \[a\ld_{\lambda(k(a))\cdot k}E\cdot k,\]
  that $\I(a/E\cdot k)$ has generators $r_1(X),\dotsc,r_N(X)$ with
  coefficients in the ring generated by $k\cup \lambda(k(a))$, after
  possibly clearing denominators. Write thus $r_i(X)=r_i(X,e_i)$, for
  polynomials $r_i(X,Z)\in k[X,Z]$ and tuples $e_i$ in
  $\lambda(k(a))$. We may assume that each $r_i$ is linear in $Z$ with
  $r_i(X,f_i)$ non-zero whenever the tuple $f_i$ from $E$ is non-zero.

  Setting now $n=\tr\bigl(a/k\cdot\lambda(k(a))\bigr)$, we may assume
  that for each $n+1$-element subtuple of $a$, one of the $r_i$'s
  witnesses that this subtuple is algebraically dependent over
  $\lambda(k(a))\cdot k$.

  Consider the tuple $\bar
  e=(e_1,\dotsc,e_N)$ as an element of $\mathbb {P}^{|e_1|}\times \cdots\times \mathbb{P}^{|e_N|}$ and denote by $\gamma(Z_1,\dotsc,Z_N)$ the locus of $\bar e$ over $k$, homogeneous in every $Z_i$. Then the
  $\LL_P$-formula \[ \chi(x) = \exists\, \zeta_1\in P \ldots \exists\,
  \zeta_N\in P \Bigl(\Land_{i=1}^N \neg\,\zeta_i\doteq 0\; \land\;
  \gamma(\zeta_1,\dotsc,\zeta_N)\;\land\; \Land\limits_{i=1}^N
  r_i(x,\zeta_i)\doteq 0 \Bigr)\] in $\tp_P(a/k)$ is equivalent to a
  tame formula, by Fact \ref{F:tame1}.

  Given a realisation $b$ of $\chi$, let $ \bar f=(f_1,\dotsc, f_N)$
  be non-trivial tuples in $E$ with $r_1(b, f_1)=\dotsb=r_N(b,
  f_n)=0$. Since the $r_i(X,f_i)$ are non-zero, it follows that
  \[\tr(b/E\cdot k)\le\tr(b/k(\bar f))\le n.\]

  If $\tr(b/E\cdot k)<n$, we have $\smallrm(b/k)<\smallrm(a/k)$. So
  let us assume from now on that $\tr(b/E\cdot k)=n$, so
  $b\ind_{k(\bar f)}E$. Corollary \ref{C:ind_lambda} yields now that
  $\lambda(k(b))\subset\acl\bigl(\lambda(k)(\bar f)\bigr)$, so
  $\tr(\lambda(k(b))/\lambda(k))\le\tr(\bar f/\lambda(k))$.
  Analogously, we have that
  \[\tr(\lambda(k(a))/\lambda(k))=\tr(\bar
  e/\lambda(k)),\] since the tuple $\bar e$ belongs to
  $\lambda(k(a))$. Since $\bar f$ satisfies $\gamma(\bar Z)$, we
  have \[\tr(\bar f/\lambda(k))\le\tr(\bar e/\lambda(k)).\]

  If the inequality is strict, we deduce that
  $\smallrm(b/k)<\smallrm(a/k)$. Assume therefore that $\tr(\bar
  f/\lambda(k))=\tr(\bar e/\lambda(k))$. We want to show that $b$ and
  $a$ have the same $\LL_P$-type over $k$. First, observe that
  $\tr(\bar f/k)=\tr(\bar e/k)$, so $\bar f$ and $\bar e$ have the
  same type over $k$, as a sequence of homogeneous tuples. By Fact
  \ref{F:Delon}, using that $k$ is \lc, we deduce that $\bar f$ and
  $\bar e$ have the same $\LL_P$-type over $k$. Hence, there exists a
  tuple $a'$ in $K$ such that $(a',\bar e)$ has the same $\LL_P$-type
  over $k$ as $(b,\bar f)$. In particular, the tuple $a'$ is a
  solution of $\I(a/k\cdot E)$ with $\tr(a'/k\cdot E)=n$. This implies
  that $a'$ and $a$ have the same type over $E\cdot k$, and thus the
  same $\LL_P$-type over $k$ by Fact \ref{F:Delon}, as desired.
  \end{proof}

\begin{cor}\label{C:rmRM}
  Given a subfield $k$ of $K$ and a finite tuple $a$, we have that
  $\RM(a/k)\le \smallrm(a/k)$.
\end{cor}

\begin{proof}
  Morley rank does not change working over independent parameters and
  neither does $\smallrm$ by Lemma \ref{L:nonfork_smallrm}. Thus, we
  may assume that $k$ is an $\aleph_0$-saturated elementary
  substructure of the pair $(K, E)$ and in particular \lc.

  For the inequality $\RM(a/k)\le \smallrm(a/k)$, it suffices to show
  inductively on $\alpha$ that $\alpha\le \smallrm(a/k)$ whenever
  $\alpha\le \RM(a/k)$. We need only consider the case
  $\alpha=\beta+1$. Let $\chi$ be a tame formula in $\tp_P(a/k)$ as in
  Proposition \ref{P:rm_isol_tame}. If $\beta+1\le \RM(a/k)$, the type
  $\tp_P(a/k)$ is an accumulation point of types over $k$ of Morley
  rank at least $\beta$. We may thus assume that there is a type
  $\tp_P(b/k)\ne \tp_P(a/k)$ of Morley rank at least $\beta$
  containing the formula $\chi$, so \[
  \smallrm(b/k)\stackrel{\ref{P:rm_isol_tame}}{<} \smallrm(a/k).\]

  By induction on $\beta\le \RM(b/k)$, we have that $\beta\le
  \smallrm(b/k)$, so $\alpha=\beta+1\le \smallrm(a/k)$, as
  desired.
\end{proof}

We deduce immediately from Lemma \ref{L:U_smallrm} as well as
Corollary \ref{C:rmRM} that Poizat's formula holds for proper pairs of
algebraically closed fields.

\begin{cor}\label{C:Poizatsformel}
  In every sufficiently saturated proper pair $(K,E)$ of algebraically
  closed fields, Morley and Lascar rank agree: Given a finite tuple
  $a$ and a subfield $k$ of $K$,
  \[ \RU(a/k)=\RM(a/k)=
  \omega \cdot \tr(a/E\cdot k) + \tr(\lambda(k(a))/\lambda(k)).\]
\end{cor}

By Fact \ref{F:Delon}, Propositions \ref{P:tame_instance} (Claim
\ref{Claim:dcl}) and \ref{P:rm_isol_tame} as well as Corollaries
\ref{C:Noether_iso_Model} and \ref{C:Poizatsformel}, we deduce the
following result:

\begin{cor}
  Given a type $p$ over an arbitrary subset $A$ of parameters of
  $(K,E)$, there exists an instance of a tame formula in $p$ which
  isolates it among all types over $A$ of rank at least $\RM(p)$. In
  particular, the theory $\acfp$ has Noetherian isolation.
\end{cor}

\section{Minimal formulae and Hilbert schemes}\label{S:Hilbertpol}
Corollary \ref{C:min} and Corollary \ref{C:acfp_noether} together
yield that every type contains a minimal tame formula. The goal of
this section, which can be seen as a complement, is to provide an
explicit description of the minimal tame formula in a given type
$\tp_P(a/k)$ whenever $k$ is \lc. (It follows from Proposition
\ref{P:tame_instance} (Claim \ref{Claim:dcl})) that the parameter set
being \lc is not an actual obstacle). For this, we will use
Grothendieck's Hilbert Scheme. Since the classical references we
consulted do not explicitly provide the construction of the Hilbert
scheme as a projective variety, we have decided to exhibit the
construction here, in Theorem \ref{T:hilbertscheme}, keeping the
exposition as elementary as possible. Most of the results from
algebraic geometry in this section can be found in \cite{EH00, dM66},
unless explicitly stated.\\

Fix a base field $F$, write $X$ for the sequence of variables $\XX$.
Given a homogeneous ideal $J$ of $F[X]$, let $J_d$ be the collection
of all homogeneous polynomials in $J$ of degree $d$. It was shown by
Hilbert \cite{dH1890} that the codimension of $J_d$ is given in terms
of a (unique) numerical polynomial $Q_J(T)$, i.e.\ having rational
coefficients, yet taking integer values when evaluated on $\Z$, called
the \emph{Hilbert polynomial} of $J$: there exists a degree $d_0$ such
that $Q_J(d)=\codim_F (J_d, F[X]_d)$ for $d\ge d_0$. It follows that
  \[\dim_F J_d=\binom{d+n}{n}-Q_J(d)\]
is also a numerical polynomial in $d$ for $d\ge d_0$.\\

\begin{definition}\label{D:Sat_Ideal}
  A homogeneous ideal $I$ of $F[X]$ is \emph{saturated} if the
  homogeneous polynomial $p(X)$ belongs to $I$ whenever $X_i^D\cdot p$
  belongs to $I$ for all $0\le i\le n$ and $D$ large enough.
\end{definition}
The homogeneous vanishing ideal over $F$ of a non-zero tuple in some
field extension of $F$ is clearly saturated. Given a homogeneous ideal
$J$ of $F[X]$, the smallest saturated ideal $\Sat(J)$ containing $J$
consists of all homogeneous polynomials $p(\bar X)$ such that for some
$D$ all the products $X_i^D\cdot p$ with $0\le i\le n$ belong to $J$.

\begin{remark}\label{R:Grad}


  Let us point out that $J_d$ and $\Sat(J)_d$ are equal for large
  enough $d$, and thus $J$ and $\Sat(J)$ have the same Hilbert
  polynomial. Indeed, the ideal $\Sat(J)$ is finitely generated by the
  polynomials $p_1,\ldots, p_m$. In particular, there exists some
  $D\ge 1$ such that $X_i^D\cdot p_j$ belongs to $J$ for all $0\le
  i\le n$. Let $d$ be greater than \[ (n+1) (D-1) +\max\limits_{1\le
    j\le m}{\deg(p_j)}. \] If $p$ belongs to $\Sat(J)_d$, then $p$ is
  a linear combination over $F$ of products of the form
  $M_\alpha(X)\cdot p_j$ for some monomials $M_\alpha(X)$ with
  $\deg(M_\alpha)= d-\deg p_j> (n+1)(D-1)$. It follows that
  $M_\alpha(X)$ contains a factor of the form $X_i^r$ with $r\ge D$
  for some $0\le i\le n$, so each $M_\alpha(X)\cdot p_j$, and hence
  $p$, belongs to $J$, as desired.

  Furthermore, two saturated ideals $J$ and $J'$ are equal whenever
  $J_d=J'_d$ for infinitely many $d's$: Indeed, if $p$ belongs to $J$,
  then $X_i^D\cdot p$ is in $J_d$ for $d\ge \deg(p)$ with
  $D=d-\deg(p)$.
\end{remark}

\begin{theorem}[Grothendiek's Hilbert scheme of $\mathbb P^n$]
  \label{T:hilbertscheme}
  For every $n$ and every numerical polynomial $Q$ the collection of
  saturated ideals in $F[X]$ with Hilbert polynomial $Q$ is in
  bijection with a projective variety $\hs(F)\subset\mathbb{P}^N(F)$
  such that, denoting by $I^{(\eta)}$ the corresponding to a tuple
  $\eta$ of $\hs(F)$ (in homogeneous coordinates), we have the
  following:
  \begin{enumerate}[(1)]
  \item\label{T:hilbertscheme:erzeugt} There is a finite set of
    polynomials $S(X,\eta)$ with integer coefficients, homogeneous
    both in $X$ and in $\eta$ which generate $I^{(\eta)}$ as a
    saturated ideal, that is, if $\scl{S(X,\eta)}$ denotes the ideal
    generated by $S(X,\eta)$, then $I^{(\eta)}=\Sat(\scl{S(X,\eta)})$.
  \item\label{T:hilbertscheme:element} For every degree $d$ there is a
    finite set $S_d(Y,\eta)$ of polynomials with integer coefficients,
    homogeneous both in $Y$ and in $\eta$, such that
    \[\sum_{|\beta|=d}c_\beta X^\beta\in I^{(\eta)}\;\Leftrightarrow\;
    S_d(c,\eta)=0\] for every tuple $c=(c_\beta)$ in $F$.
  \end{enumerate}
The variety $\hs$ as well as the sets $S$ and $S_d$ are all defined
over $\Z$ and do not depend on the field $F$.
\end{theorem}
\begin{proof}
  We fix $n$ and $Q$ as in the statement. By a result of Mumford
  (\cite[Theorem III-55 \& Page 263]{EH00} \& \cite[Lecture 14]{dM66})
  there is a degree $d_0$ such that for every \emph{saturated} ideal
  $I$ of $F[X]$ with Hilbert polynomial $Q$ we have
  \[\dim_F I_d=\binom{d+n}{n}-Q(d)\quad
  \text{ for all } d\ge d_0.\] \noindent Moreover, the ideal
  $\bigoplus\limits_{d\ge d_0} I_d$ is generated by $I_{d_0}$. This
  has the following consequence.
  \begin{claim}\label{Claim:Corresp_Hilbert}
    The assignment $I\mapsto U=I_{d_0}$ defines a bijection between
    all saturated ideals $I$ of $F[\XX]$ with Hilbert polynomial $Q$
    and the set $\Hs$ of all subspaces $U$ of $F[\XX]_{d_0}$
    satisfying \[
    \label{E:Gleichung}
    \dim_F\,\scl{U}_d=\binom{d+n}{n}-Q(d) \quad\text{ for all }d\ge
    d_0 \tag*{(\ding{249})},\] where $\scl U$ is the ideal generated
    by $U$.
  \end{claim}

  \noindent Note that all elements of $\Hs$ have the same dimension,
  namely $N_0=\binom{d_0+n}{n}-Q(d_0)$.\\

  \begin{claimproof}
    If $I$ is saturated with Hilbert polynomial $Q$, then $U=I_{d_0}$
    belongs to $\hs$, since $\scl{U}_d=I_d$ for all $d\ge d_0$. This
    also shows that $I=\Sat\scl U$ is uniquely determined by $U$.
    Conversely, if $U$ belongs to $\hs$, it follows that $I=\Sat(\scl
    U)$ has Hilbert polynomial $Q$, by Remark \ref{R:Grad}. In
    particular, \[ \dim_F I_{d_0}=N_0 = \dim_F U,\] whence
    $I_{d_0}=U$, as desired.
  \end{claimproof}

  We will see below that the collection of $N_0$-dimensional subspaces
  of $F[X]_{d_0}$ form a projective variety $\Gr_{N_0}(F[X]_{d_0})$,
  called the \emph{$N_0^{th}$-Grassmannian}. Let $\hs$ be the subset
  of $\Gr_{N_0}(F[X]_{d_0})$ which corresponds to $\Hs$. Together with
  Claim \ref{Claim:Corresp_Hilbert}, Claim \ref{Claim:hilbert_schema}
  below yields that $\hs$ is a Zariski closed subset of
  $\Gr_{N_0}(F[X]_{d_0})$ definable without parameters (or defined
  over $\Z$ in algebraic terms). It is easy to see that every suitable
  choice of the degree $d_0$ yields the same variety $\hs$, up to
  canonical isomorphism.

  Fix some $s$ in $\N$. To view the set of all $r$-dimensional
  subspaces $V$ of the vector space $F^s$ as a projective variety, we
  will encode $V$ by the exterior product $v_1\wedge\dotsb\wedge v_r$,
  where $v_1,\ldots, v_r$ is some basis of $V$. Up to a non-zero
  scalar factor this exterior product only depends on $V$, so it
  determines a unique element of the projective space $\mathbb
  P(\extp^{r} F^s)$. Its (homogeneous) coordinates are the
  \emph{Pl\"ucker coordinates} $\pk(V)$ of $V$. Given Pl\"ucker
  coordinates $\pk(V)=v_1\wedge\dotsb\wedge v_r$ in $\mathbb
  P(\extp^{r} F^s)$, we recover $V$ as the set of all vectors $v$ in
  $F^s$ such that $v\wedge(v_1\wedge\dotsb\wedge v_r)=0$. The
  collection $\Gr_r(F^s)$ of Pl\"ucker coordinates $\eta$ is given by
  the quadratic equations $\eta\wedge(e^*\pluck\eta)=0$, where $e^*$
  runs through some basis of $\extp^{r-1}(F^s)^*$ and the map \[
  \pluck:\extp^{r-1} (F^s)^* \times \extp^{r} (F^s) \to F^s\] is the
  (bilinear) inner product of exterior algebra (see \cite[Page
    182]{jD74}). Thus, the $r^{th}$-Grassmannian $\Gr_r(F^s)$ of $F^s$
  is a projective variety.\\

  Note that $\binom{s}{r}=\dim\extp^r(F^s)$, so
  $\Gr_{N_0}(F[X]_{d_0})$ is a subvariety of $\mathbb P(F^N)$ with
  $N=\binom{\binom{d_0+n}{n} } {N_0}$.

  We will now use the following result of Macaulay \cite{fsM27} (of which Sperner
  gave a simplified proof in \cite{eS30}): \emph{If $Q$ is the
    Hilbert polynomial of some homogeneous ideal in $n+1$ variables
    and $d_1$ is large enough, then for any homogeneous ideal $J$ of
    $F[X]$ with
  \[\dim J_{d_1}\geq\binom{d_1+n}{n}-Q(d_1),\]
  it follows that
  \[\dim J_d\geq\binom{d+n}{n}-Q(d)\quad\text{ for all } d\geq d_1.\]
  }

  \begin{claim}\label{Claim:hilbert_schema}
    For $d_0$ large enough, the set $\hs$ as defined above is a
    Zariski closed subset of $\Gr_{N_0}F[X]_{d_0}$, definable without
    parameters.
  \end{claim}

  \begin{claimproof}
    If no ideal of $F[X]$ has $Q$ as Hilbert polynomial, then $\hs$ is
    empty and thus Zariski closed. Otherwise, we may assume that $Q$
    satisfies Macaulay's property with respect to $d_1$, which we may
    assume to be equal to $d_0$ as in Mumford's result. It now follows
    that an $N_0$-dimensional subspace $U$ of $F[X]_{d_0}$ satisfies
    the condition \eqref{E:Gleichung} of Claim
    \ref{Claim:Corresp_Hilbert} if and only if
    \[\dim_F\,\scl{U}_d\le\binom{d+n}{n}-Q(d) \quad\text{ for all }d\ge
    d_0.\] If $\eta$ are the Pl\"ucker coordinates of $U$, the
    polynomials $e^*\pluck\eta$, where $e^*$ runs among the elements
    of the canonical basis of $\extp^{N_0-1}(F[X]_{d_0})^*$, generate
    $U$ as an $F$--vector space \cite[R\'esultats d'Alg\`ebre (A,44)
      p.183, IX]{jD74}. Thus, the graded component $\scl{U}_d$ of the
    ideal $\scl U$ is generated as an $F$--vector space by the
    polynomials $M_\alpha \cdot(e^*\pluck\eta)$, where the
    $M_\alpha$'s enumerate all monomials in the variables $X$ of
    degree $d-d_0$. Hence, the Pl\"ucker coordinates $\eta$ belong to
    $\hs$ if and only if for all $d\geq d_0$ the dimension of the
    vector space generated by the $M_\alpha\cdot(e^* \pluck\eta)$ is
    bounded by $\binom{d+n}{n}-Q(d)$. The last condition can be
    expressed by determinantal equations in the coefficients of
    $\eta$, as desired.
  \end{claimproof}

  Property (\ref{T:hilbertscheme:erzeugt}) follows easily from the
  proof of the last claim, if one defines $S(X,\eta)$ as the set of
  all $(e^* \pluck\eta)$, where $e^*=e^*(X)$ are elements of the
  canonical basis of $\extp^{N_0-1}(F[X]_{d_0})^*$.\\

  For property (\ref{T:hilbertscheme:element}), fix some degree $d$
  and a polynomial $h(c,X)=\sum_{|\beta|=d}c_\beta X^\beta$ of degree
  $d$. If $d<d_0$, note that $h(c,X)$ belongs to $I^{(\eta)}$ if and
  only if all $X_i^{d_0}\cdot h(c,X)$ belong to $I^{(\eta)}$, for
  $I^{(\eta)}$ is saturated. Therefore, we may assume that the
  polynomial $h(c,X)$ has degree $d\ge d_0$, so $h(c,X)$ belongs to
  $I^{(\eta)}$ if and only if the $F$-vector space generated by
  $h(x,X)$ and all $M_\alpha\cdot(e^*\pluck\eta)$ where the $M_\alpha$
  are monomials of degree $d-d_0$ has dimension at most
  $\binom{d+n}{n}-Q(d)$. Again, the latter can be expressed by
  determinantal equations $S_d(c, \eta)=0$ in $c$ and $\eta$.
\end{proof}
\bigskip

We will see now how Theorem \ref{T:hilbertscheme} produces the minimal
tame formula in a given type $\tp_P(a/k)$, where $k$ is a \lc subfield
of a sufficiently saturated model $(K,E)$ and $a=(a_1,\dotsc,a_n)$ is
a tuple in $K$.

The homogeneous vanishing ideal $I$ of the tuple $(1,a_1,\dots,a_n)$
over the subfield $F=E\cdot k$ is saturated, so denote by $Q$ its
Hilbert polynomial. By Theorem \ref{T:hilbertscheme}, let $\hs$ be the
corresponding Hilbert scheme as well as $S$ and $S_d$ the associated
finite sets of polynomials. The ideal $I$ is of the form
$I=I^{(\eta)}$ for some $\eta$ in $\hs$. Clearing denominators, we may
assume that the coefficients of $\eta$ belong to the ring generated by
$E\cup k$. As in the proof of Proposition \ref{P:Trick_proof}, write
$\eta=\eta(e)$ for some is a tuple $\eta(Z)$ of linear forms over $k$
and a non-trivial tuple $e$ in $E$. Rewriting the coefficients of
$\eta$ with respect to a fixed basis of $k$ over $\lambda(k)$, which
remains linearly independent over $E$ by linear disjointness, we may
assume that $\eta(f)\ne 0$ for every non-zero tuple $f$ in $E$.

Denoting by $\gamma(\zeta)$ the homogeneous locus of $e$ over $k$, we
have the following:

\begin{theorem}\label{T:min_fmla_Hilb}
  The tame formula
  \[\phi(x)=\exists\zeta\in P\;\Bigl(
  \neg\,\zeta\doteq 0\;\land\;\gamma(\zeta)\;\land\;\hs(\eta(\zeta))
  \;\land\; S(1,x_1,\dotsc,x_n,\eta(\zeta))\doteq 0\Bigr)\] is the
  minimal tame formula in $\tp_P(a/k)$.
\end{theorem}

\begin{proof}
  It is clear that the tame formula $\phi(x)$ belongs to $\tp_P(a/k)$,
  setting $\zeta=e$, since $S(X,\eta(e))$ is contained in
  $I=I^{\eta(e)}$.

 Assume now that we are given an instance of a tame formula
 \[\psi(x)=\exists\;\zeta'\in P\;\bigl(\neg\,\zeta'\doteq
 0\;\land\;\Land_{j=1}^m q_j(1,x,\zeta')\doteq 0\bigr)\] in
 $\tp_P(a/k)$ for some polynomials $q_\ell$ in $k[X,Z']$, homogeneous
 both in $X$ and in $Z'$. By Theorem \ref{T:hilbertscheme}
 (\ref{T:hilbertscheme:element}) there is for every degree $\deg(q_j)$
 a finite set $S_{q_j}(Z',Z)$ of polynomials over $k$, homogeneous
 both in $Z$ and in $Z'$, such that for all tuples $f$ and $f'$ in $E$
 with $\eta(f)$ in $\hs$, then
 \[q_j(X,f')\in I^{\eta(f)}\;\Leftrightarrow\;S_{q_j}(f',f)=0.\]

 Now, the tuple $a$ realises $\psi$, so there exists a non-trivial
 tuple $e'$ in $E$ such that $q_j(1,a, e')=0$ for all $1\le j\le m$ .
 Therefore, each polynomial $q_j(X,e')$ belongs to $I^{(\eta(e))}$, so
 $S_{q_j}(e',e)=0$, by the above. Hence, the tuple $e$ satisfies the
 tame formula
 \[\rho(\zeta)=\exists\;\zeta'\in E\;\bigl(\neg\,\zeta'\doteq
 0\;\land\;\Land_{j=1}^m S_{q_j}(\zeta',\zeta)\doteq 0\bigr).\] As $k$
 is \lc, Lemma \ref{L:tame=Zar} yields that for realisations in $E$
 the formula $\rho(\zeta)$ is equivalent to a finite system of
 equations $\sigma(\zeta)$ with coefficients in $k$, homogeneous in
 $\zeta$. In particular, every solution in $E$ of the homogeneous
 locus $\gamma(\zeta)$ of $e$ over $k$ is a solution of the system
 $\sigma(\zeta)$, and thus satisfies $\rho(\zeta)$.

 Given now a realisation $b$ of $\phi(x)$, we need to show that $b$
 also satisfies $\psi(x)$, which gives that $\phi$ is the minimal tame
 formula in $\tp_P(a/k)$. Since $b$ realises $\phi(x)$, there is
 non-zero tuple $f$ in $E$ which is a solution of $\gamma(\zeta)$ such
 that $\eta(f)$ lies in $\hs$ and $S(1,b,\eta(f))= 0$.

 As $f$ is a solution of $\gamma(\zeta)$, the above discussion yields
 that $f$ realises $\rho(\zeta)$, so there is a non-zero tuple $f'$ in
 $E$ such that \[ \Land_{j=1}^m S_{q_j}(f',f)= 0.\] We deduce that
 every polynomial $q_j(X,f')$ belongs to $I^{\eta(f)}$. Since
 $S(1,b,\eta(f))=0$ and the polynomials in $S(X,\eta(f)$ generate
 $I^{\eta(f)}$ as a saturated ideal, it follows that $q_j(1,b,f')=0$
 for all $1\le j\le m$, so $b$ realises $\psi(x)$, as desired.
\end{proof}

\end{document}